\documentclass[reqno, a4paper, 11pt]{amsart}
\usepackage{amsfonts, amsmath, amssymb, amsthm, color}
\usepackage[hmargin=3cm, vmargin=3cm]{geometry}
\usepackage[varg]{txfonts}
\usepackage{graphicx, setspace}
\usepackage{hyperref}

\newtheorem{thm}{Theorem}[section]
\newtheorem{cor}[thm]{Corollary}
\newtheorem{lem}[thm]{Lemma}
\newtheorem{prop}[thm]{Proposition}

\newtheorem{innercustomthm}{Theorem}
\newenvironment{customthm}[1]
  {\renewcommand\theinnercustomthm{#1}\innercustomthm}
  {\endinnercustomthm}

\theoremstyle{definition}

\newtheorem{rem}[thm]{Remark}

\numberwithin{equation}{section}
\allowdisplaybreaks

\newcommand{\ep}{\epsilon}
\newcommand{\lx}{{\lambda, \xi}}
\newcommand{\pa}{\partial}

\newcommand{\wtu}{\widetilde{U}}

\newcommand{\drn}{D^{1,2}(\mr_+^{N+1}; t^{1-2s})}
\newcommand{\hc}{H^{1,2}_0(\mc; t^{1-2s})}

\newcommand{\mc}{\mathcal{C}}
\newcommand{\mh}{\mathcal{H}}
\newcommand{\mt}{\mathcal{T}}
\newcommand{\mv}{\mathcal{V}}
\newcommand{\mn}{\mathbb{N}}
\newcommand{\mr}{\mathbb{R}}

\renewcommand{\(}{\left(}
\renewcommand{\)}{\right)}

\begin{document}
\title[{Classification of solutions to the fractional Lane-Emden-Fowler equations}]{Classification of finite energy solutions\\ to the fractional Lane-Emden-Fowler equations\\ with slightly subcritical exponents}

\author{Woocheol Choi}
\address[Woocheol Choi]{School of Mathematics, Korea Institute for Advanced Study, Seoul 130-722, Republic of Korea}
\email{wchoi@kias.re.kr}

\author{Seunghyeok Kim}
\address[Seunghyeok Kim]{Facultad de Matem\'{a}ticas, Pontificia Universidad Cat\'{o}lica de Chile, Avenida Vicu\~{n}a Mackenna 4860, Santiago, Chile}
\email{shkim0401@gmail.com}

\begin{abstract}
We study qualitative properties of solutions to the fractional Lane-Emden-Fowler equations with slightly subcritical exponents
where the associated fractional Laplacian is defined in terms of either the spectra of the Dirichlet Laplacian or the integral representation.
As a consequence, we classify the asymptotic behavior of all finite energy solutions.
Our method also provides a simple and unified approach to deal with the classical (local) Lane-Emden-Fowler equation for any dimension greater than 2.
\end{abstract}

\thanks{\textit{Acknowledgement.} W. Choi is grateful to the financial support from POSCO TJ Park Foundation.
S. Kim is supported by FONDECYT Grant 3140530.}

\maketitle
\setcounter{tocdepth}{1}
\tableofcontents

\section{Introduction}
Suppose that $s \in (0,1)$, $N > 2s$, $p = \frac{N+2s}{N-2s}$ and $\Omega$ is a smooth bounded domain.
In this paper we are concerned with the asymptotic behavior of solutions to the nonlinear nonlocal elliptic problem
\begin{equation}\label{eq-main}
\begin{cases}
(-\Delta)^s u = u^{p-\ep} &\text{in } \Omega,\\
u > 0 &\text{in } \Omega,\\
u = 0 &\text{on } \Sigma = \pa \Omega \text{ or } \mr^N \setminus \Omega,
\end{cases}
\end{equation}
when a small parameter $\ep >0$ tends to zero.
Here $(-\Delta)^s$ is understood as the spectral fractional Laplacian or the restricted fractional Laplacian according to the choice of the boundary $\Sigma = \pa \Omega$ or $\mr^N \setminus \Omega$, respectively
(see Subsection \ref{subsec-def-frac} for the definition of the fractional Laplacians).

\medskip
Recently various nonlocal differential equations have attracted lots of researchers.
Especially, equations involving the fractional Laplacian were treated extensively in both pure and applied mathematics,
because not only the fractional Laplacian is an operator which naturally interpolates the classical Laplacian $-\Delta$ and the identity $(-\Delta)^0 = id$,
but also it appears in diverse areas including physics, biological modeling and mathematical finances, as a tool describing nonlocal characteristic.

Owing to technical difficulties arising from the nonlocality, there had not been enough progress in theory of equations involving the fractional Laplacian.
However, about a decade ago, Caffarelli and Silvestre \cite{CS} interpreted the fractional Laplacian in $\mr^N$ in terms of a Dirichlet-Neumann type operator in the extended domain $\mr^{N+1}_{+} = \{(x,t) \in \mr^{N+1}: t > 0\}$,
and this idea allowed one to analyze nonlocal problems by utilizing well-known arguments such as the mountain pass theorem, the moving plane method, the Moser iteration, monotonicity formulae, etc.
A similar extension was also devised by Cabr\'e-Tan \cite{CT}, and Stinga-Torrea \cite{ST}
(see Capella-D\'{a}vila-Dupaigne-Sire \cite{CDDS}, Br\"andle-Colorado-de Pablo-S\'anchez \cite{BCPS1}, Tan \cite{T2} and Chang-Gonz\'alez \cite{CG} also)
for nonlocal elliptic equations on bounded domains with zero Dirichlet boundary condition.

Based on these extensions (or the integral representation of a differential operator itself), a lot of studies on nonlocal problems of the form $(-\Delta)^s u = f(u)$ (for a certain function $f: \mr \to \mr$) were conducted.
For the results of particular equations, we refer to papers on the Schr\"odinger equations \cite{FQT, DDW, CZ, AWY}, the Allen-Cahn equations \cite{CS2, CS22},
the Fisher-KPP equations \cite{BCRR, CR}, the Nirenberg problem \cite{AC, JLX, JLX2}, and the Yamabe problem \cite{GQ, GW, CK, KMW}, respectively.
Also, Brezis-Nirenberg type problems have been tackled in \cite{T, BCPS2, ChS}.
Most results mentioned here considered on the existence of solutions with some desired property.
Meanwhile, several regularity results such as the Schauder estimate and the strong maximum principle were derived in \cite{CT, ST, CDDS, CS2, JLX, CS3} and references therein.

\medskip
Due to its simple form, the Lane-Emden-Fowler problem \eqref{eq-main} has been regarded as one of the most fundamental nonlinear elliptic equations.
It is now a classical fact that the exponent $p = \frac{N+2s}{N-2s}$ is a threshold on the existence of a solution to \eqref{eq-main}.
If $\ep > 0$, one can find a solution to \eqref{eq-main} by applying the standard variational argument with the compact embedding $H^s(\Omega) \hookrightarrow L^{p+1-\ep}(\Omega)$.
If $\ep \le 0$ and $\Omega$ is star-shaped, the Pohozaev identity (obtained in \cite{CT, T2} for the spectral Laplacians and in \cite{RS3} for the restricted Laplacians) implies that no solution exists.
In view of the corresponding result of Bahri-Coron \cite{BC} to the case $s = 1$,
it is expected that \eqref{eq-main} has a solution if the domain $\Omega$ has nontrivial topology.

On the other hand, it is well-known that the Brezis-Nirenberg type problem
\begin{equation}\label{eq-BN}
\begin{cases}
(-\Delta)^s u = u^p + \ep u^q &\text{in } \Omega,\\
u > 0 &\text{in } \Omega,\\
u = 0 &\text{on } \Sigma = \pa \Omega \text{ or } \mr^N \setminus \Omega,
\end{cases}
\end{equation}
where $N > 2s$, $0 < q < p$ and $\ep > 0$ is a parameter, shares many common characteristics with \eqref{eq-main}.
Through the papers \cite{T, BCPS2, SV2, BCSS}, it was determined that its solvability relies on $\ep,\,p,\, q,\, N$ and $\Omega$.

\medskip
Once the existence theory is settled, the very next step would be to obtain information on the shape of solutions.

For equation \eqref{eq-main} with general exponents on the nonlinearity, an answer of this question is provided by the moving plane argument.
It yields that for any $p-\ep > 1$ each solution to \eqref{eq-main} increases along lines emanating from a boundary point to a certain interior point.
It then induces symmetry of a solution from that of the domain $\Omega$. We refer to \cite{CT, RS, T2} for further discussion.

On the other hand, it is natural to guess that if $\ep \to 0$, then the solution $u_{\ep}$ may possess a singular behavior, since $p = {N+2s \over N-2s}$ is the critical exponent.
This idea intrigues one to investigate the shape of $u_{\ep}$ in detail for $\ep > 0$ small enough.
In this regards, Choi-Kim-Lee \cite{CKL} and D\'avila-L\'opez-Sire \cite{DLS} constructed multiple blow-up solutions
by applying the Lyapunov-Schmidt reduction method (refer to Theorem \ref{thm-m-multipeak2} below).
When the fractional Laplacian is defined in terms of the spectra of the Dirichlet Laplacian,
the authors of \cite{CKL} also characterized the asymptotic behavior of a sequence $\{u_{\ep}\}_{\ep >0}$ of \textit{minimal energy} solutions to \eqref{eq-main} and \eqref{eq-BN} (with $q = 1$).
It turned out that $u_{\ep}$ blows up at a single point which is a critical point of the Robin function of $(-\Delta)^s$.

In this line of research, an important remaining problem is to study the asymptotic character of solutions $\{u_{\ep}\}_{\ep >0}$ without the minimal energy condition.
This is what we address in the current paper.
Precisely, we shall give a detailed description for the asymptotic behavior of all \textit{finite energy} solutions to \eqref{eq-main} where the fractional Laplacian is either spectral or restricted one.
We believe that the same phenomena should happen to finite energy solutions to \eqref{eq-BN}.

\begin{thm}\label{thm-ch}
For any given $s \in (0,1)$ and $N > 2s$, suppose that there exists a sequence $\{u_n\}_{n \in \mn}$ in $\mh$ such that each of the function $u_n$ solves equation \eqref{eq-main} with $\ep = \ep_n \searrow 0$.
In addition, assume $\sup_{n \in \mn} \|u_n\|_\mh < + \infty$. Then one of the following holds: Up to a subsequence, either

\noindent \textnormal{(1)} the function $u_n$ converges strongly in $\mh$ to a function $v$ satisfying
\begin{equation}\label{eq-l}
\begin{cases}
(-\Delta)^s v = v^p &\text{in } \Omega,\\
v > 0 &\text{in } \Omega,\\
v = 0 &\text{on } \Sigma = \pa \Omega \ \text{ or } \mr^N \setminus \Omega
\end{cases}
\end{equation}
as $n \to \infty$, or

\noindent \textnormal{(2)} the asymptotic behavior of $u_n$ is given by \begin{equation}\label{eq-asym}
u_n = \sum_{i=1}^m Pw_{\lambda_n^i, x_n^i} + r_n
\end{equation}
where $\lambda_n^i \to 0$ and $x_n^i \to x_0^i \in \Omega$ as $n \to \infty$.
Here $Pw_\lx$ is the projected bubble defined after \eqref{eq-pw} and $r_n$ is a remainder term converging to zero in $\mh$.
Furthermore the following properties are valid.

\noindent - There is a constant $C_0 > 0$ independent of $n \in \mn$ such that $\frac{\lambda_n^j}{\lambda_n^i} < C_0$ holds for all $n \in \mn$ and $i, j = 1, \cdots, m$.

\noindent - There is a constant $d_0 > 0$ such that $|x_n^i - x_n^j| > d_0$ for any $n \in \mn$ and $i,j = 1, \cdots, m$ with $i \neq j$.

\noindent - Let $b_i = \( \lim_{n \to \infty} \frac{\lambda_n^i}{\lambda_n^1}\)^{\frac{N-2s}{2}}$ and $b_0 = \lim_{n \to \infty} (\lambda_n^1)^{-(N-2s)} \ep_n$.
Then the value
\[((b_1, \cdots, b_m), (x_0^1, \cdots, x_0^m)) \subset (0,\infty)^m \times \Omega^m\]
is a critical point of the function $\Phi_m$ defined by
\begin{equation}\label{eq-phi}
\Phi_m (b_1,\cdots, b_m, x_1,\cdots, x_m)
= c_1 \( \sum_{i=1}^m b_i^2 H(x_i, x_i) - \sum_{i \neq k} b_i b_k G(x_i, x_k)\) - c_2 \log (b_1 \cdots b_m) \cdot b_0,
\end{equation}
where
\begin{equation}\label{eq-c12}
c_1 = \int_{\mr^N} w_{1,0}^p dx > 0 \quad \text{and} \quad c_2 = \({N-2s \over N}\) {\int_{\mr^N} w_{1,0}^{p+1} dx \over \int_{\mr^N} w_{1,0}^p dx} > 0.
\end{equation}
Here $G: \Omega \times \Omega \to \mr$ is Green's function of $(-\Delta)^s$,
$H: \Omega \times \Omega \to \mr$ is its regular part, and $w_{1,0}$ is the standard bubble on $\mr^N$ given in \eqref{eq-bubble}.
(See Section 2 for more details.)
\end{thm}
\begin{rem}
As we mentioned, equation \eqref{eq-main} may have a solution even for $\ep \le 0$ if the topology of the domain $\Omega$ is not simple (say, its homology group over $\mathbb{Z}/(2\mathbb{Z})$ is non-trivial).
Hence the first case (1) of Theorem \ref{thm-ch} cannot be excluded for general domains.
\end{rem}

If the blow-up points satisfy a certain non-degeneracy condition, then we can determine the blow-up rates in terms of an explicit power of $\ep^{-1}$ as the following theorem shows.

\begin{thm}\label{thm-main2}
Let $\{u_n\}_{n \in \mn}$ be a sequence of solutions to \eqref{eq-main} satisfying \textnormal{(2)} of Theorem \ref{thm-ch}.
Let us set an $m \times m$ symmetric matrix $M = (m_{ij})_{1 \leq i, j \leq m}$ by
\[m_{ij} = \begin{cases}
H(x_0^i, x_0^i)  &\text{if}~ i =j,\\
-G(x_0^i, x_0^j) &\text{if}~ i \neq j.
\end{cases}\]
Then it is nonnegative definite. If it is nondegenerate (i.e. positive definite), then for any $1 \leq i \leq m$, we have
\begin{equation}\label{eq-conc-rate}
\lim_{n \to \infty} \log_{\ep_n} \lambda_n^i = \frac{1}{N-2s}.
\end{equation}
\end{thm}

\medskip
Recall that equation \eqref{eq-main} has multi-bubble solutions as the following result indicates.
\begin{customthm}{A}[Choi-Kim-Lee \cite{CKL} and D\'avila-L\'opez-Sire \cite{DLS}]\label{thm-m-multipeak2}
Assume $s \in (0,1)$ and $N > 2s$.
Given arbitrary $m \in \mn$, suppose that the function $\Phi_m$ in \eqref{eq-phi} with $b_0 = (N-2s)/4s$ has a stable critical set $\Lambda_m$
such that
\[\Lambda_m \subset \left\{((\lambda_1, \cdots, \lambda_m), (x_1,\cdots,x_m)) \in (0,\infty)^m \times \Omega^m : x_i \ne x_j \text{ if } i \ne j \text{ and } i, j = 1, \cdots, m\right\}.\]
Then there exist a point $((\lambda_0^1, \cdots, \lambda_0^m),(x_0^1,\cdots,x_0^m)) \in \Lambda_m$ and a small number $\ep_0 > 0$ such that for $0 < \ep < \ep_0$,
there is a family of solutions $u_{\ep}$ of \eqref{eq-main} which concentrate at each point $x_0^1, \cdots, x_0^{m-1}$ and $x_0^m$ as $\ep \to 0$ in the form \eqref{eq-asym},
after extracting a subsequence if necessary.
\end{customthm}
\noindent The asymptotic behavior of solutions figured in Theorem \ref{thm-main2} (2) corresponds exactly to the multi-peak solutions described in the above theorem.
This reveals the accuracy and sharpness of our classification results.
The question of finding a blow-up sequence of solutions not satisfying \eqref{eq-conc-rate} is open even for the local case $s = 1$.

\medskip
Before introducing our strategy for the proof of the classification results, it is worth to remind that problem \eqref{eq-main} is a nonlocal version of the Lane-Emden-Fowler equation
\begin{equation}\label{eq-local}
\begin{cases}
-\Delta u = u^{\frac{N+2}{N-2}-\ep} &\text{in } \Omega, \\
u>0 &\text{in } \Omega, \\
u=0 &\text{on } \pa \Omega.
\end{cases}
\end{equation}
In \cite{R}, Rey constructed one-peak solutions to \eqref{eq-local}.
Then multi-peak solutions were found by Bahri-Li-Rey \cite{BLR}, Rey \cite{R3} and Musso-Pistoia \cite{MP} (for $N \ge 3$) by different ways.
Furthermore, the classification of solutions was conducted in Han \cite{H} and Rey \cite{R} for one-peak case ($N \ge 3$), and in Bahri-Li-Rey \cite{BLR} and Rey \cite{R3} for general case ($N \ge 4$ and $N = 3$, respectively).
\begin{customthm}{B}[Bahri-Li-Rey \cite{BLR} and Rey \cite{R3}] \label{thm-local}
Assume that $N \ge 3$ and $\{u_n\}_{n \in \mn} \subset H^1_0(\Omega)$ is a sequence of solutions to \eqref{eq-local} with $\ep = \ep_n \searrow 0$.
Also, suppose that $\sup_{n \in \mn} \|u_n\|_{H^1_0(\Omega)} < \infty$.

\noindent \textnormal{(1)} Passing to a subsequence, either $u_n$ strongly converges to a solution $u$ of \eqref{eq-l} with $s = 1$,
or it has the asymptotic behavior \eqref{eq-asym} where $Pw_\lx$ is the projected bubble defined as
\[-\Delta Pw_\lx = w_\lx^p \quad \text{in } \Omega \quad \text{and} \quad Pw_\lx = 0 \quad \text{on } \pa \Omega\]
($w_{\lx}$ is given in \eqref{eq-bubble}).
Moreover, all characteristics of the concentration points $\{x_n^1, \cdots, x_n^m\}$ and rates $\{\lambda_n^1, \cdots, \lambda_n^m\}$ in the statement of Theorem \ref{thm-ch} remain to hold.
If the nonnegative matrix $M$ defined in the statement of Theorem \ref{thm-main2} is in fact positive, then \eqref{eq-conc-rate} is valid.
\end{customthm}

\noindent In \cite{BLR, R3}, a certain decomposition of the space $H^1_0(\Omega)$ is crucially used (see Remark \ref{rem-intro} (1) below), which produces large error in the lowest dimension case $N = 3$.
In this reason, improved estimates had to be made additionally in \cite{R3}.
Remarkably, as we will see later, our proof for Theorems \ref{thm-ch} and \ref{thm-main2} provides a unified and neater approach to treat this local situation $s = 1$.
As a result, we have a new proof of Theorem \ref{thm-local} working for all dimensions $N \ge 3$ at the same time. See Subsection \ref{subsec-local}.

\medskip
The framework of the proofs for our main theorems comprises of the following three steps:

\noindent \textbf{Step 1.} Concentration-compactness principle;

\noindent \textbf{Step 2.} Pointwise bounds of $u_n$ obtained from a moving sphere argument and their applications;

\noindent \textbf{Step 3.} Two identities regarding Green's function and the Robin function coming from a type of Green's identity.

Let us briefly explain each step by assuming that the spectral fractional Laplacian is under consideration.

In \textbf{Step 1}, we recall the concentration-compactness principle for problem \eqref{eq-main}.
This renowned principle is found by Struwe \cite{S} for equation \eqref{eq-local}, 
and recently extended to problem \eqref{eq-main} by Almaraz \cite{A} for $s=\frac{1}{2}$, and by Fang-Gonz\'alez \cite{FG} and Palatucci-Pisante \cite{PP} for all $0<s<1$ (in slightly different setting).
It makes possible to decompose solutions $\{u_n\}_{n \in \mn}$ of \eqref{eq-main} as
\begin{equation}\label{eq-cc-local}
u_n = v_0 + \sum_{i=1}^m Pw_{\lambda_n^i, x_n^i} + r_n,
\end{equation}
where $v_0$ is the $\mh$-weak limit of $\{u_n\}_{n \in \mn}$, $Pw_{\lambda^i_n,x^i_n} \in \mh$ is the projected bubble and $r_n$ converges to zero in $\mh$.
See Lemma \ref{lem-cc-bounded} for the complete description of $\lambda_n^i$, $x_n^i$, $v_0$, $Pw_\lx$ and $r_n$.

Now our task is reduced to getting further information on the sequence $\{u_n\}_{n \in \mn}$ whose elements are expressed as \eqref{eq-cc-local},
which is one of the main contributions of this paper.
We immediately encounter a difficulty, because we do not know at this moment even whether two different concentration points $x_n^i$ and $x_n^j$ may collide or not.
This technicality will be tackled in \textbf{Step 2}, where we attain a pointwise bound of $u_n$ near each concentration point by employing the moving sphere method towards the extended problem \eqref{eq-ext} of equation \eqref{eq-main} (see Section \ref{sec-moving}).
This allows us to deduce no coincidence of two different blow-up points and to obtain further valuable information on solutions
such as the alternative between $v_0=0$ and $m=0$, and compatibility of blow-up rates of all peaks (see Section \ref{sec-point}).
This part is motivated by Schoen \cite{Sch}.

Given the pointwise bound and its consequences derived in \textbf{Step 2}, we show in \textbf{Step 3} that the $L^{\infty}$-normalized sequence of the solutions $u_n$ converges to a combination of Green's functions.
Then inserting this information into a Green-type identity \eqref{eq-uv} will lead us to discover two identities \eqref{eq-con} and \eqref{eq-con2}
regarding on the limit of the blow-up profile $(\lambda_n^1, \cdots, \lambda_n^m, x_n^1, \cdots x_n^m)$,
which will complete the proof of our main results.
On passing to the limit, one needs to know a uniform $C^2$-estimate of the $s$-harmonic extensions of $\{u_n\}_{n \in \mn}$.
It is not a trivial issue since we are handling the nonlocal problem \eqref{eq-main}, or the associate degenerate local problem \eqref{eq-ext} with the weighted Neumann boundary condition.
Appendix \ref{sec-app-b} is devoted to deduce the desired regularity results.

The above strategy extends Han's method \cite{H} in a quite natural manner, while the argument in Bahri-Li-Rey \cite{BLR} and Rey \cite{R3} can be regarded as further developments of Rey \cite{R, R2}.

\medskip
We conclude this section, presenting some additional remarks.
\begin{rem}\label{rem-intro}
(1) The corresponding result to \textbf{Step 3} for the local problem \eqref{eq-local} was achieved in Bahri-Li-Rey \cite{BLR} and Rey \cite{R3}.
The argument in \cite{BLR, R3} requires one to estimate $\|r_n\|_{H_0^1(\Omega)}$ in terms of powers of $\ep_n$ and $\max_{1\leq k \leq m} \lambda_n^k$.
For this aim, the authors replaced $\sum_{i=1}^m Pw_{\lambda_n^i, x_n^i}$ in the expansion \eqref{eq-cc-local} of $u_n$ with $\sum_{i=1}^m \alpha_n^i Pw_{\lambda_n^i, x_n^i}$ (for some $\alpha_n^i \in \mr$)
and then perturbed the parameters $(\alpha_n^i, \lambda_n^i, x_n^i)$ so that $r_n$ satisfies the $H^1_0(\Omega)$-orthogonality
\[\langle r_n, Pw_{\lambda_n^i, x_n^i} \rangle_{H^1_0(\Omega)} = \left\langle r_n, \frac{\pa Pw_{\lambda_n^i, x_n^i}}{\pa x_j} \right\rangle_{H^1_0(\Omega)}
= \left\langle r_n, \frac{\pa Pw_{\lambda_n^i, x_n^i}}{\pa \lambda} \right\rangle_{H^1_0(\Omega)} = 0
\quad \text{for } 1 \leq i \leq m, ~1 \leq j \leq N,\]
as in Bahri-Coron \cite{BC}.
After that, they followed the argument of Rey \cite{R, R2} to get a sharp estimate $\|r_n\|_{H_0^1 (\Omega)}$.
Their argument is simplified in our proof in the point that we do not need the estimate of the remainder term $r_n$.

\medskip \noindent (2) An advantage of the argument in \cite{BLR, R3} is that it deals with the energy functional of \eqref{eq-local} directly
so that it suggests a way to compute the Morse index of the solutions.
Recently, asymptotic behavior of the first $(N+2)m$-eigenvalues and eigenfunctions for the linearized equation of \eqref{eq-local} was examined in \cite{GP, CKL2}.
They give the information on the Morse index as a particular corollary.
\end{rem}

The rest of this paper is organized as follows.
In Section \ref{sec-pre}, we review the extension problem for the spectral and restricted fractional Laplacians, Green's function, the Robin function and the projected bubbles.
Moreover, we recall the concentration-compactness principle which brings with a decomposition result of blow-up solutions.
Section \ref{sec-moving} is devoted to the proof of a pointwise upper bound which makes use of a moving sphere argument.
In Section \ref{sec-point}, by using this estimate, we attain various refined information for the blow-up solutions, and in particular, show that suitably normalized blow-up solutions converge to combinations of Green's functions.
In Section \ref{sec-bi}, we obtain essential information of the blow-up points and their blow-up rates by using a Green-type identity, which proves our main results.
For the sake of brevity, we concentrate only on the spectral fractional Laplacian in Sections \ref{sec-moving}-\ref{sec-bi}.
Instead, all necessary modifications to deal with the restricted fractional Laplacian or the classical (local) Laplacian are listed in Section \ref{sec-res}.
Finally, a decay estimate of the standard bubble $W_{1,0}$ (see Subsection \ref{subsec_sobolev_trace}) needed in Section \ref{sec-moving}
and elliptic regularity results necessary for Lemma \ref{lem-u-asym} are derived in Appendices \ref{sec-W} and \ref{sec-app-b}, respectively.

\medskip
\noindent \textbf{Notations.}

\medskip \noindent - The letter $z$ represents a variable in the half-space $\mr^{N+1}_+ = \mr^N \times (0,\infty)$. Also, it is written as $z = (x,t) = (x_1, \cdots, x_N, x_{N+1})$ with $x = (x_1, \cdots, x_N) \in \mr^N$ and $t = x_{N+1} > 0$.

\medskip \noindent - For any fixed smooth bounded domain $\Omega \subset \mr^N$, let $\mc := \Omega \times (0, \infty) \subset \mr^{N+1}_+$ be the associated cylinder of $\Omega$ and $\pa_L \mc := \pa \Omega \times (0,\infty)$ its lateral boundary.
Set also $\mc' := \Omega \times [0, \infty)$.

\medskip \noindent - For fixed $N \in \mn$ and $s \in (0,1)$ such that $N > 2s$,
the weighted Sobolev space $\drn$ is defined as the completion of the space $C^{\infty}_c(\overline{\mr^{N+1}_+})$ with respect to the norm
\[\|U\|_{\drn} := \(\int_{\mr_+^{N+1}} t^{1-2s} |\nabla U(z)|^2 dz\)^{1/2} \quad \text{for } U \in C^{\infty}_c(\overline{\mr^{N+1}_+}).\]
Moreover, for any given cylinder $\mc = \Omega \times (0,\infty)$ (where $\Omega$ is a smooth bounded domain),
the space $\hc$ is the completion of $C^{\infty}_c(\mc \cup (\Omega \times \{0\}))$ with respect to the above norm.

\medskip \noindent - We will denote by $p$ the critical exponent $\frac{N+2s}{N-2s}$.

\medskip \noindent - Let $B^{N+1}_+((x,0),r)$ be the half-ball in $\mr^{N+1}_+$ of radius $r$ centered at $(x,0) \in \mr^N \times \{0\}$.
Moreover, we set $\pa_I B^{N+1}_+(0,r) = \pa B^{N+1}_+(0,r) \cap \mr^{N+1}$.

\medskip \noindent - $dS$ stands for the surface measure. Also, a subscript attached to $dS$ (such as $dS_x$ or $dS_z$) denotes the variable of the surface.

\medskip \noindent - For an arbitrary domain $D \subset \mr^n$, the map $\nu = (\nu_1, \cdots, \nu_n): \pa D \to \mr^n$ denotes the outward unit normal vector on $\pa D$.

\medskip \noindent - Suppose that $D$ is a domain and $T \subset \pa D$. If $f$ is a function on $D$, then the trace of $f$ on $T$ is denoted by $\text{tr}|_T f$ whenever it is well-defined.

\medskip \noindent - $|S^{N-1}| = 2\pi^{N/2}/\Gamma(N/2)$ denotes the Lebesgue measure of $(N-1)$-dimensional unit sphere $S^{N-1}$.

\medskip \noindent - The following positive constants will appear in \eqref{eq-frac}, \eqref{eq-CS}, \eqref{eq-Poi}, \eqref{eq-Green-R}, \eqref{eq-bubble} and \eqref{eq-Sobolev}:

\[c_{N,s} := {2^{2s} s\Gamma({N+2s \over 2}) \over \pi^{N \over 2} \Gamma(1-s)},
\quad \kappa_s := \frac{\Gamma(s)}{2^{1-2s}\Gamma (1-s)},
\quad p_{N,s} := {\Gamma\({N+2s \over 2}\) \over \pi^{N \over 2}\Gamma(s)},
\quad \gamma_{N,s} := \frac{1}{|S^{N-1}|}\cdot \frac{2^{1-2s} \Gamma\(\frac{N-2s}{2}\)}{\Gamma\(\frac{N}{2}\) \Gamma(s)},\]
\[\alpha_{N,s} := 2^{\frac{N-2s}{2}} \( \frac{\Gamma \( \frac{N+2s}{2}\)}{\Gamma \( \frac{N-2s}{2}\)}\)^{\frac{N-2s}{4s}}
\quad \text{and} \quad \mathcal{S}_{N,s} := 2^{-s}\pi^{-{s \over 2}} \(\frac{\Gamma\(\frac{N-2s}{2}\)}{\Gamma\(\frac{N+2s}{2}\)}\)^{\frac{1}{2}} \(\frac{\Gamma(N)}{\Gamma({N \over 2})}\)^{\frac{s}{N}}.\]

\medskip \noindent - $C > 0$ is a generic value that may vary from line to line.


\section{Preliminaries on Fractional Laplacians} \label{sec-pre}
In this section we review some preliminary notions and results which will be needed throughout the proofs of the main theorems.

\subsection{Definition of Sobolev Spaces and Fractional Laplacians} \label{subsec-def-frac}
For any smooth bounded domain $\Omega$, let $\{ \lambda_k, \phi_k \}_{k=1}^{\infty}$ be a sequence of the eigenvalues
and the corresponding $L^2(\Omega)$-normalized eigenvectors of the Dirichlet Laplacian $-\Delta$ in $\Omega$,
\[\begin{cases}
-\Delta \phi_k = \lambda_k \phi_k \quad \text{in } \Omega \quad \text{and} \quad \phi_k = 0 \quad \text{on } \pa \Omega,\\
\|\phi_k\|_{L^2(\Omega)} = 1
\end{cases}\]
where $0 < \lambda_1 < \lambda_2 \le \lambda_3 \le \cdots$.
Introduce a space
\[\mv^s(\Omega) = \left\{u = \sum_{i=1}^{\infty} a_i \phi_i \in L^2(\Omega): \|u\|^2_{\mathcal{V}^s(\Omega)} := \sum_{i=1}^{\infty} a_i^2 \lambda_i^{2s} < \infty \right\}.\]
Then the \textit{spectral Laplacian} is defined as
\[(-\Delta)^s u = \sum_{i=1}^{\infty} a_i \lambda_i^{2s} \phi_i \quad \text{for any } u = \sum_{i=1}^{\infty} a_i \phi_i \in \mv^s(\Omega).\]
It is known that
\[\mv^s(\Omega) = \left\{u = \text{tr}|_{\Omega \times \{0\}} U: U \in \hc \right\} = \begin{cases}
H^s(\Omega) &\text{for } 0 < s < 1/2,\\
H^s_{00}(\Omega) &\text{for } s = 1/2,\\
H^s_0(\Omega) &\text{for } 1/2 < s < 1
\end{cases}\]
where $H^s(\Omega)$ is the usual fractional Sobolev space,
$H^s_0(\Omega)$ is the closure of $C_c^{\infty}(\Omega)$ with respect to the Sobolev norm $\|\cdot\|_{H^s(\Omega)}$ and
\[H_{00}^{1/2}(\Omega) := \left\{u \in H^{1/2}(\Omega): \int_\Omega {u(x)^2 \over \text{dist}(x,\pa \Omega)}\, dx < \infty\right\}\]
(refer to \cite{CaL}).

On the other hand, for any $s \in (0,1)$ and $u \in H^s(\mr^N)$, we are capable of defining the fractional Laplacian by using the integral representation
\begin{equation}\label{eq-frac}
(-\Delta)^s u(x) = c_{N,s}\, \text{P.V.}\, \int_{\mr^N} \frac{u(x)-u(y)}{|x-y|^{N+2s}} dy.
\end{equation}
Here the exact value of $c_{N,s} > 0$ (as well as other constants such as $\kappa_s$ or $p_{N,s}$ below) can be found at the last part of the previous section.
If this operator is restricted to functions in $H^s_0(\Omega)$, then it is called the \textit{restricted fractional Laplacian}.

To compare two different fractional Laplacians, the reader is advised to check the papers \cite{MN, SV, BSV}.

We set
\begin{equation}\label{eq-space}
\mh = \begin{cases}
\mv^s(\Omega) &\text{if the spectral fractional Laplacian is concerned},\\
H^s_0(\Omega) &\text{if the restricted fractional Laplacian is concerned.}
\end{cases}
\end{equation}

\begin{rem}
At the first glance, the boundary condition of \eqref{eq-main}, that is, $u = 0$ in $\pa \Omega$ for $0 < s < 1/2$ may be ambiguous because $H_0^s(\Omega) = H^s(\Omega)$.
However, elliptic regularity guarantees that $u$ is bounded, so the representation formula makes sense.
It is continuous up to the boundary and has zero boundary values.
\end{rem}

\subsection{Localization of Fractional Laplacians}
For a fixed function $u \in \mv^s(\Omega)$ (or $H^s(\mr^N)$), let us set $U \in \hc$ (or $\drn$, respectively) to be the $s$-harmonic extension of $u$, namely, a unique solution of the equation
\[\begin{cases}
\text{div}(t^{1-2s} \nabla U) = 0 &\text{in } \mc \quad (\text{or } \mr^{N+1}_+),\\
U = 0 &\text{on } \pa_L \mc \quad (\text{or } \pa_L \mr^{N+1}_+ = \emptyset),\\
U(\cdot,0)= u &\text{on } \Omega \quad (\text{or } \mr^N).
\end{cases}\]
Then by the celebrated results of Caffarelli-Silvestre \cite{CS} (for the Euclidean space $\mr^N$) and Cabr\'e-Tan \cite{CT} (for bounded domains $\Omega$, see also \cite{ST, CDDS, T2}), it holds that
\begin{equation}\label{eq-CS}
(-\Delta)^s u (x)= \pa_{\nu}^s U (x) := - \kappa_s \lim_{t \to 0+} t^{1-2s} \frac{\pa U}{\pa t}(x,t) \quad \text{for } x \in \Omega \text{ (or } \mr^N).
\end{equation}
Moreover, if $u \in H^s(\mr^N)$, then the Poisson representation formula gives that
\begin{equation}\label{eq-Poi}
U(x,t) = p_{N,s} \int_{\mr^N} {t^{2s} \over (|x-y|^2+t^2)^{N+2s \over 2}}\, u(y)\, dy
\end{equation}
while for $u \in \mv^s(\Omega)$ it is possible to describe $U$ in terms of a series (refer to \cite{CDDS}).

\medskip
As a result, if the spectral fractional Laplacian is concerned, then the $s$-harmonic extension $U_{\ep} \in \hc$ of a
solution $u_{\ep} \in \mv^s(\Omega)$ to problem \eqref{eq-main} satisfies
\begin{equation}\label{eq-ext}
\begin{cases}
\text{div}(t^{1-2s} \nabla U_{\ep}) = 0 &\text{in } \mc, \\
U_{\ep} = 0 &\text{on } \pa_L \mc,\\
U_{\ep} = u_{\ep} &\text{on } \Omega \times \{0\},\\
\pa_{\nu}^s U_{\ep} = u_{\ep}^{p-\ep} &\text{on } \Omega \times \{0\}.
\end{cases}
\end{equation}
In light of the Sobolev inequality \eqref{eq-Sobolev}, we see
\begin{equation}\label{eq-U_ep}
\|U_\ep\|_{\hc}^2 = \|u_\ep\|_{L^{p+1-\ep}(\Omega)}^{p+1-\ep} \le C \|u_\ep\|_{\mv^s(\Omega)}^{p+1-\ep}.
\end{equation}
Therefore if we have $\sup_{\ep > 0} \|u_\ep\|_{\mv^s(\Omega)} < + \infty$,
then $\sup_{\ep > 0} \|U_\ep\|_{\hc} < + \infty$.
Moreover, by the strong maximum principle (\cite[Corollary 4.12]{CS2} or \cite[Lemma 2.7]{FW}), it holds that $U_{\ep} > 0$ in $\mc$.

A similar (and in fact simpler) formulation is available when the restricted fractional Laplacian is studied.
In this case, the equation we have to consider is
\begin{equation}\label{eq-ext-2}
\begin{cases}
\text{div}(t^{1-2s} \nabla U_{\ep}) = 0 &\text{in } \mr^{N+1}_+, \\
U_{\ep} = 0 &\text{on } (\mr^N \setminus \Omega) \times \{0\},\\
U_{\ep} = u_{\ep} &\text{on } \Omega \times \{0\},\\
\pa_{\nu}^s U_{\ep} = u_{\ep}^{p-\ep} &\text{on } \Omega \times \{0\}.
\end{cases}
\end{equation}

\subsection{Green's Functions of Fractional Laplacians}
In this subsection, we review Green's functions.

\medskip
We consider first the case when the fractional Laplacian is defined in terms of the spectra of the Laplacian. We refer to \cite{CKL} for more details.

Let $G$ be Green's function of the the spectral fractional Laplacian $(-\Delta)^s$ on a smooth bounded domain $\Omega$ with the zero Dirichlet boundary condition.
Then it can be regarded as the trace of Green's function $G_{\mc} = G_{\mc}(z,x)$ ($z \in \mc$, $x \in \Omega$) for the Dirichlet-Neumann problem on the extended domain $\mc$ which satisfies
\begin{equation}\label{eq-Green}
\begin{cases}
\text{div} (t^{1-2s} \nabla G_{\mc}(\cdot, x)) =0 &\text{in } \mc,\\
G_{\mc}(\cdot, x) = 0 &\text{on } \pa_L \mc,\\
\pa_{\nu}^{s} G_{\mc}(\cdot, x) = \delta_x &\text{on } \Omega\times \{0\}
\end{cases}
\end{equation}
where $\delta_x$ is the Dirac delta function on $\mathbb{R}^n$ with center at $x \in \Omega$.

Green's function $G_{\mc}$ on the half-cylinder $\mc$ can be decomposed into the singular and regular parts.
The singular part is given by Green's function
\begin{equation}\label{eq-Green-R}
G_{\mr^{N+1}_{+}}((x,t),y) := \frac{ {\gamma_{N,s}}}{|(x-y,t)|^{N-2s}} \end{equation}
on the half-space $\mr^{N+1}_{+}$ satisfying
\begin{equation}\label{eq-prop-G}
\begin{cases}
\text{div} \(t^{1-2s} \nabla_{(x,t)} G_{\mr^{N+1}_{+}} ((x,t),y)\) = 0 &\text{in } \mr^{N+1}_+,\\
\pa_{\nu}^{s} G_{\mr^{N+1}_{+}} ((x,0),y)= \delta_{y}(x) &\text{on } \mr^N = \pa \mr^{N+1}_+
\end{cases}
\end{equation}
for each $y \in \mr^N$.
The regular part is given as the function $H_{\mc} : \mc \to \mr$ which solves
\begin{equation}\label{eq-prop-H}
\begin{cases}
\text{div}\(t^{1-2s} \nabla_{(x,t)}H_{\mc}((x,t),y) \)= 0 &\text{in } \mc,\\
H_{\mc}((x,t),y) = \dfrac{\gamma_{N,s}}{|(x-y,t)|^{N-2s}} &\text{on } \pa_L \mc,\\
\pa_{\nu}^{s}H_{\mc} ((x,0),y) = 0 &\text{on } \Omega \times \{0\}
\end{cases}
\end{equation}
for any $y \in \Omega$.
Its existence can be verified in a variational method (see Lemma 2.2 in \cite{CKL}).
We then have
\[G_{\mc}((x,t),y) = G_{\mr^{N+1}_{+}} ((x,t),y) - H_{\mc}((x,t),y).\]
Now, letting $H(x,y) = H_{\mc} ((x,0),y)$, we can decompose $G(x,y) = G_{\mc}((x,0),y)$ as follows.
\[G(x,y) = {\gamma_{N,s} \over |x-y|^{N-2s}} - H(x,y).\]

Let us recall some regularity properties of the function $H$. For any index $\alpha \in (\mn \cup \{0\})^N$, the partial derivatives $\pa_{x}^{\alpha} H_{\mc}$ of $H_{\mc}$ in the $x$-variable always exist (see Lemma \ref{lem-high-2} and Section 2 of \cite{CKL}). In addition, it follows from \eqref{eq-prop-H} that
\[\begin{cases}
\text{div}\(t^{1-2s} \nabla_{(x,t)}\pa_x^{\alpha}H_{\mc}((x,t),y) \)= 0 &\text{in } \mc,\\
\pa_{\nu}^{s} \pa_x^{\alpha} H_{\mc} ((x,0),y)=0 &\text{on } \Omega \times \{0\}.
\end{cases}\]
Therefore, by applying \cite[Lemma 4.5]{CS2} to each $\pa_{x}^{\alpha} H_{\mc}$, we see that there is a constant $C = C(\alpha, r,\xi) > 0$ such that
\begin{equation}\label{eq-prop-h1}
|\pa_{x}^{\alpha} H_{\mc}((x,t),y)| \leq C
\end{equation}
and
\begin{equation}\label{eq-prop-h2}
\left| t^{1-2s}\pa_t \pa_{x}^{\alpha} H_{\mc}((x,t),y)\right| \leq C
\end{equation}
for all $(x,t) \in B_+^{N+1}((\xi,0), r)$ provided that $\xi \in \Omega$ and $r > 0$ satisfy the condition $r < \text{dist}(\xi,\pa \Omega)$.

\medskip
When the restricted fractional Laplacian is dealt with, we observe that the above discussion is still valid
once we let $\mc = \mr^{N+1}_+$ and substitute the boundary conditions in \eqref{eq-Green} and \eqref{eq-prop-H} with
\[G_{\mc}(\cdot, x) = 0 \quad \text{on } \pa_B\mc \quad \text{and} \quad H_{\mc}((x,t),y) = {\gamma_{N,s} \over |(x-y,t)|^{N-2s}} \quad \text{on } \pa_B\mc\]
respectively, where $\pa_B \mc := (\mr^N \setminus \Omega) \times \{0\}$.
(The function $G_{\mc}$ in this paragraph should not be confused with the fundamental solution $G_{\mr^{N+1}_+}$ in \eqref{eq-Green-R}.)

\subsection{Sharp Sobolev and Trace Inequalities} \label{subsec_sobolev_trace}
Given any $\lambda > 0$ and $\xi \in \mr^N$, let $w_\lx$ be the \textit{bubble} defined by
\begin{equation} \label{eq-bubble}
w_\lx (x) = \alpha_{N,s} \( \frac{\lambda}{\lambda^2+|x-\xi|^2}\)^{\frac{N-2s}{2}} \quad \text{for } x \in \mr^N.
\end{equation}
Then it is true that
\begin{equation}\label{eq-Sobolev}
\( \int_{\mr^N} |u|^{p+1} dx \)^{\frac{1}{p+1}} \leq \mathcal{S}_{n,s} \( \int_{\mr^N} |(-\Delta)^{s/2} u|^2 dx \)^{1 \over 2},
\end{equation}
and the equality holds if and only if $u(x) = c w_{\lambda,\xi}(x)$ for any $c > 0,\, \lambda>0$ and $\xi \in \mr^N$ (refer to \cite{L2, CaL, FL}).
Furthermore, it was shown in \cite{CLO, L1, L3} that if a suitable decay assumption is imposed, then $\{ w_\lx: \lambda>0, \xi \in \mr^N \}$ is the set of all solutions for the problem
\[(-\Delta)^s u = u^p,\quad u > 0 \quad \text{in } \mr^N\quad \text{and}\quad \lim_{|x|\to \infty} u(x) = 0.\]

Denote also the $s$-harmonic extension of $w_{\lambda,\xi}$ by $W_{\lambda,\xi} \in \drn$ so that $W_\lx$ solves
\begin{equation}\label{wlyxt}
\begin{cases}
\text{div}(t^{1-2s}W_{\lambda,\xi}(x,t)) = 0 &\text{in } \mr^{N+1}_+,\\
W_{\lambda,\xi}(x,0) = w_{\lambda,\xi}(x) &\text{on } \mr^N.
\end{cases}
\end{equation}
It follows that for the Sobolev trace inequality
\begin{equation}\label{eq-sharp-trace}
\( \int_{\mr^N} |U(x,0)|^{p+1} dx \)^{\frac{1}{p+1}} \leq \sqrt{\kappa_s}\, \mathcal{S}_{n,s} \( \int_0^{\infty}\int_{\mr^N} t^{1-2s} |\nabla U(x,t)|^2 dx dt \)^{1 \over 2},
\end{equation}
the two sides are equal if and only if $U(x,t) = c W_{\lambda,\xi}(x,t)$ for any $c > 0,\ \lambda>0$ and $\xi \in \mr^N$.

\subsection{Concentration-Compactness Principle} \label{subsec-cc}
Firstly, we treat the spectral fractional Laplacian case.
Let $PW_\lx$ stand for the projection of the bubble $W_\lx$ into $\hc$, that is, the solution of
\begin{equation}\label{eq-pw}
\begin{cases}
\text{div}(t^{1-2s} \nabla PW_\lx) = 0 &\text{in}~\mc,
\\
PW_\lx =0 &\text{on}~\pa_L \mc,
\\
\pa_{\nu}^s PW_\lx = \pa_{\nu}^s W_\lx = W_\lx^p &\text{on}~\Omega \times \{0\},
\end{cases} \end{equation}
and $Pw_\lx = \text{tr}|_{\Omega \times \{0\}} PW_\lx$.
By the maximum principle \cite[Lemma 2.1]{CKL}, we have $0 \le PW_\lx \le W_\lx$ in $\mc$.
Also \cite[Lemma C.1]{CKL} says that
\begin{equation}\label{eq-pw-exp}
PW_\lx(z) = W_\lx(z) - c_1 \lambda^{N-2s \over 2} H(z,\sigma) + o(\lambda^{N-2s \over 2})
\end{equation}
uniformly for $z \in \mc$ where $c_1 > 0$ is the number appeared in \eqref{eq-phi}.

The following result is a fractional version of Struwe \cite{S}.
\begin{lem}\label{lem-cc-bounded}
Let $\{ U_n \}_{n \in \mn}$ be a sequence of solutions to \eqref{eq-ext} with $\ep =\ep_n \searrow 0$ which satisfies the norm condition $\sup_{n \in \mn} \|U_n\|_{\hc} < \infty$.
Then there exist an integer $m \in \mn \cup \{0\}$ and a sequence $\{(\lambda_n^i, x_n^i )\}_{n \in \mn} \subset (0,\infty) \times \Omega$ of positive numbers and points for each $i = 1,\cdots, m$ such that
\begin{equation}\label{eq-cc}
R_n := U_n - \(V_0 + \sum_{i=1}^{m} PW_{\lambda_n^i, x_n^i}\) \to 0
\text{ in } \hc \quad \text{as } n \to \infty
\end{equation}
(up to a subsequence) where $V_0$ is the weak limit of $U_n$ in $\hc$, which satisfies
\begin{equation}\label{eq-V-eq}
\begin{cases}
\textnormal{div}(t^{1-2s} \nabla V_0 ) = 0 &\text{in } \mc,\\
V_0 = 0 &\text{on } \pa_L \mc,\\
\pa_{\nu}^s V_0 = V_0^{\frac{N+2s}{N-2s}} &\text{on } \Omega \times \{0\}.
\end{cases}
\end{equation}
In addition, it holds that
\begin{equation}\label{eq-cc-cond}
{1 \over \lambda_n^i}\, \textnormal{dist}(x_n^i, \pa \Omega) \to \infty
\quad \text{and} \quad
\frac{\lambda_n^i}{\lambda_n^j} + \frac{\lambda_n^j}{\lambda_n^i} + {1 \over \lambda_n^i \lambda_n^j}\, |x_n^i -x_n^j|^2 \to \infty \quad \text{as}~ n \to \infty
\end{equation}
for all $1 \le i \neq j \le m$.
\end{lem}
\begin{proof}
See \cite{A} and \cite{FG} where an analogous conclusion is deduced in the setting of asymptotically hyperbolic manifolds.
Since their approach still works for our case, we omit the proof.
\end{proof}

Let $v_0 = \text{tr}|_{\Omega \times \{0\}} V_0$ and $r_n = \text{tr}|_{\Omega \times \{0\}} R_n$.

Extracting a subsequence of $\{U_n\}_{n \in \mn}$ and reordering the indices if necessary, we may assume that
\begin{equation}\label{eq-lam-order}
\lambda_n^1 \le \lambda_n^2 \le \cdots \le \lambda_n^m \quad \text{for all } n \in \mn \quad \text{and} \quad x_n^i \to x_0^i \in \overline{\Omega} \quad \text{as } n \to \infty.
\end{equation}
Using the Kelvin transform and the moving plane argument, Choi \cite[Lemma 4.1]{Ch} proved that $\{U_n\}_{n \in \mn}$ are uniformly bounded near the boundary $\pa \Omega \times[0,\infty)$.
That is, there exists constants $\delta >0$ and $C>0$ such that
\[\sup_{n \in \mn} \, \sup_{\{(x,t) \in \mc: \text{dist}(x, \pa \Omega) <\delta\}} |U_n(x,t)| \leq C.\]
Hence
\begin{equation}\label{eq-x_i}
\text{dist}(x_0^i, \pa \Omega) \ge \delta \quad \text{for } i = 1, \cdots, m.
\end{equation}

\medskip
For the restricted fractional Laplacian, we define $PW_{\lx}$ by \eqref{eq-pw} whose second line is replaced with $PW_{\lx} = 0$ in $\mr^N \setminus \Omega$.
Then it is not hard to draw analogous results to Lemma \ref{lem-cc-bounded} (cf. \cite{PP}) and \eqref{eq-pw-exp}.
Besides one can check that \eqref{eq-x_i} still holds as follows:
If the domain $\Omega$ is strictly convex, we apply the moving plane method with the maximum principle for small domains (given in \cite[Lemma 5.1]{RS2}), getting
\begin{equation}\label{eq-x_i-2}
\sup_{n \in \mn} \, \sup_{\text{dist}(x, \pa \Omega) <\delta} |u_n(x)| \leq C.
\end{equation}
In the case that $\Omega$ does not have the convexity assumption, we first use the conformal invariance of equation \eqref{eq-main}
(refer to \cite[Proposition A.1]{RS}) and then employ the moving plane method to obtain \eqref{eq-x_i-2}.
Now combining \eqref{eq-cc} and \eqref{eq-x_i-2} gives \eqref{eq-x_i} at once.
See \cite[Section 2]{H} to recall the argument used for the local case $s = 1$.

\medskip
In the next two sections, further information on blow-up rates $\{\lambda_n^i\}_{i=1}^m$ and points $\{x_n^i\}_{i=1}^m$ in the decomposition \eqref{eq-cc} will be examined.
In what follows, we simply denote $w_{1,0}$ and $W_{1,0}$ by $w$ and $W$, respectively.
Since $W = W(x,t)$ is radially symmetric in the $x$-variable, we will often write $W(x,t) = W(\rho, t)$ where $\rho = |x|$.
In addition, the operator $(-\Delta)^s$ is understood as the \textit{spectral} fractional Laplacian (and hence $\Sigma = \pa \Omega$ in equation \eqref{eq-main}) in Sections \ref{sec-moving}, \ref{sec-point} and \ref{sec-bi}.
Consideration on the \textit{restricted} fractional Laplacian is postponed to Section \ref{sec-res}.

\section{Moving Sphere Argument and Pointwise Upper Bound} \label{sec-moving}
The aim of this section is to obtain a sharp pointwise upper bound of solutions $U_{\ep}$ to \eqref{eq-ext}. To this end, we will employ the method of moving spheres (refer to \cite{Sch, ChaC, LZ}).

\begin{prop}\label{prop-moving}
Let $r_0 > 0$ be any fixed small number.
Assume that $\{M_{\ep}\}_{\ep > 0}$ is a family of positive numbers such that $\lim_{\ep \to \infty} M_{\ep} = \infty$ and $\lim_{\ep \to \infty} M_{\ep}^{\ep}=1$.
If a family $\{V_{\ep}\}_{\ep > 0}$ of positive functions which satisfy
\begin{equation}\label{eq-a-eq1}
\begin{cases}
\textnormal{div}(t^{1-2s} \nabla V_{\ep})=0 &\textnormal{in } B^N \Big(0, r_0M_{\ep}^{\frac{2}{N-2s}}\Big) \times (0,\infty),
\\
\pa_{\nu}^s V_{\ep} = V_{\ep}^{p-\ep} &\textnormal{on } B^N \Big(0, r_0M_{\ep}^{\frac{2}{N-2s}}\Big),\\
\|V_{\ep}\|_{L^{\infty}\Big(B^{N+1}_+\Big(0,r_0M_{\ep}^{\frac{2}{N-2s}}\Big)\Big)} \le c
\end{cases}
\end{equation}
for some $c > 0$, and
\begin{equation}\label{eq-a-eq2}
V_{\ep} \rightharpoonup W \text{ weakly in } \drn \quad \text{as } \ep \to 0,
\end{equation}
then there are constants $C > 0$ and $0 < \delta_0 < r_0$ independent of $\ep > 0$ such that
\[V_{\ep} (z) \le C W(z) \quad \text{for all } z \in B^{N+1}_+\Big(0, \delta_0 M_{\ep}^{\frac{2}{N-2s}}\Big).\]
\end{prop}

For the proof of the above proposition, we make some remarks.
\begin{rem}\label{rem-moving}
(1) By \eqref{eq-a-eq1}, \eqref{eq-a-eq2} and the H\"older regularity due to Cabre-Sire \cite{CS2},
if a constant $\zeta_1 > 0$ and a compact set $K \subset \overline{\mr^{N+1}_+}$ are given,
then there exist $\ep_1 > 0$ small and $\alpha \in (0,1)$ such that
\begin{equation}\label{eq-V_ep-W}
\|V_{\ep} - W\|_{C^{\alpha}(K)} \le \zeta_1 \quad \text{for } \ep \in (0, \ep_1).
\end{equation}

\medskip \noindent (2) For any function $F$ in $\mr^{N+1}_+$, let $F^{\lambda}$ be its Kelvin transform of defined as
\begin{equation} \label{eq-Kelvin}
F(z) = \(\frac{\lambda}{|z|}\)^{N-2s} F\big(z^{\lambda}\big)
\quad \text{where } z^{\lambda} := \frac{\lambda^2 z}{|z|^2} \in \mr^{N+1}_+.
\end{equation}
If we write $D_{\ep}^{\lambda} = V_{\ep} - V_{\ep}^{\lambda}$, then it holds that
\[\pa_{\nu}^s D_{\ep}^{\lambda} = V_{\ep}^{p-\ep} - \({\lambda \over |x|}\)^{(N-2s)\ep} \(V_{\ep}^{\lambda}\)^{p-\ep}
\ge V_{\ep}^{p-\ep} - \(V_{\ep}^{\lambda}\)^{p-\ep} = \xi_{\ep}(x)\, D_{\ep}^{\lambda} \quad \text{for } |x| \ge \lambda \text{ and } t = 0\]
where
\[\xi_{\ep} (x) = \begin{cases}
\dfrac{V_{\ep}^{p-\ep}-\(V_{\ep}^{\lambda}\)^{p-\ep}}{V_{\ep} - V_{\ep}^{\lambda}}(x,0) &\text{if } V_{\ep}(x,0) \ne V_{\ep}^{\lambda}(x,0),\\
(p-\ep) V_{\ep}^{p-1-\ep}(x,0) &\text{if } V_{\ep}(x,0) = V_{\ep}^{\lambda}(x,0).
\end{cases}\]

\medskip \noindent (3) For each $R > 0$, let us introduce Green's function $G^R$ of the spectral fractional Laplacian $(-\Delta)^s$ in $\Omega = B^N(0,R)$ with zero Dirichlet boundary condition
and Green's function $G_{\mc}^{R}$ of equation \eqref{eq-Green} in the cylinder $\mc = B^N (0,R) \times (0,\infty)$.
By the scaling invariance, we have
\[G^R (x,y) = \frac{1}{R^{N-2s}} G^1\( \frac{x}{R}, \frac{y}{R}\) \quad \text{for}~x, y \in B^{N}(0,R)\]
and
\[G_{\mc}^{R} ((x,t),y) = \frac{1}{R^{N-2s}} G_{\mc}^1 \(\(\frac{x}{R},\frac{t}{R}\), \frac{y}{R}\) \quad \text{for } x, y \in B^N(0,R) \text{ and } t > 0.\]
Thus we can decompose Green's function in $B^N(0,R)$ into its singular part and regular part as follows:
\begin{equation}\label{eq-G-R}
G_{\mc}^R ((x,t),y) = {\gamma_{N,s} \over |(x-y,t)|^{N-2s}} - {1 \over R^{N-2s}}H_{\mc}^1 \(\({x \over R},\frac{t}{R}\),{y \over R}\) \quad \text{for } x, y \in B^N(0,R),~ t > 0.
\end{equation}
The precise value of the normalizing constant $\gamma_n$ is given in Notations.
\end{rem}

As a preliminary step, we prove the minimum of $V_{\ep}$ on any half-sphere $\{z \in \mr^{N+1}_+: |z| = r\}$ is controlled by the value $W(r,0)$
whenever $r$ is at most of order $M_{\ep}^{2 \over N-2s}$ and $\ep > 0$ is small enough.
\begin{lem}\label{lem-moser-1}
Let $\{V_{\ep}\}_{\ep > 0}$ be the family in the statement of Proposition \ref{prop-moving}.
Then, for any $\zeta_2 > 0$, there exist small constants $\delta_1 \in (0, r_0)$ and $\ep_2 > 0$ such that
\begin{equation}\label{eq-inf-bound}
\min_{\{z \in \mr^{N+1}_+: |z|=r\}} V_{\ep} (z) \le (1+ \zeta_2) W(r,0) \quad \text{for any } 0 < r \le \delta_1 M_{\ep}^{\frac{2}{N-2s}} \text{ and } \ep \in (0, \ep_2).
\end{equation}
\end{lem}
\begin{proof}
The proof is divided into 3 steps.

\medskip \noindent \textsc{Step 1.}
We assert that for any parameter $0 < \lambda < 1$, there exists a large number $R = R(\lambda)> 0$ such that
\begin{equation}\label{eq-W-l}
(W - W_{\lambda^2,0})(z) > 0 \quad \text{for } \lambda < |z| \le R.
\end{equation}
A direct computation with \eqref{eq-bubble} shows that $w^{\lambda}(x) = w_{\lambda^2,0}(x)$ for any $\lambda > 0$ and $x \in \mr^N$.
By \cite[Proposition 2.6]{FW} and the uniqueness of the $s$-harmonic extension, it follows that $W^{\lambda} = W_{\lambda^2,0}$ in $\mr^{N+1}_+$.
Hence \eqref{eq-Kelvin} and \eqref{eq-app} imply that
\[\begin{cases}
\text{div}(t^{1-2s} \nabla (W - W_{\lambda^2,0})) = 0 &\text{in } \mr^{N+1}_+,\\
(W - W_{\lambda^2,0})(z) = (W - W^{\lambda})(z) = 0 &\text{on } |z| = \lambda \text{ and } t > 0,\\
(W - W_{\lambda^2,0})(z) > 0 &\text{on } |z| = R \text{ and } t > 0,\\
(W - W_{\lambda^2,0})(x,0) = (w - w_{\lambda^2,0})(x) > 0  &\text{on } \lambda < |x| \le R
\end{cases}\]
for some $R > 0$ large.
Now the (classical) strong maximum principle justifies our claim \eqref{eq-W-l}.

We also notice that
\begin{equation}\label{eq-W-rel}
W(x,t) \le w(x) \le w(0) = \alpha_{N,s} \quad \text{for } (x,t) \in \mr^{N+1}_+
\end{equation}
where $\alpha_{N,s} > 0$ is given in Notations.

\medskip \noindent \textsc{Step 2.}
From the definition \eqref{eq-Kelvin} we have
\begin{equation}
V_{\epsilon}^{\lambda}(z) = \( \frac{\lambda}{|z|}\)^{N-2s} V_{\epsilon}\( \frac{\lambda^2 z}{|z|^2}\).
\end{equation}
By \eqref{eq-V_ep-W} and \eqref{eq-W-rel}, there are values $\eta_1 > 0$ small and $R_0 > 0$ large such that
\begin{equation}\label{eq-u_k-3}
V_{\ep}^{\lambda}(z) \le \(1+{\zeta_2 \over 4}\) \alpha_{N,s} |z|^{-(N-2s)} \quad \text{for any } 0 < \lambda \le 1+\eta_1 \text{ and } |x| \le R_0,
\end{equation}
provided $\ep > 0$ small enough.
Let us take $\lambda_1= 1-\eta_1$ and $\lambda_2 = 1+\eta_1$.
Thanks to estimates \eqref{eq-V_ep-W} and \eqref{eq-W-l}, it is possible to select numbers $\eta_2 > 0$ small and $R_1 > R_0$ large such that
\begin{equation}\label{eq-w_k}
\begin{aligned}
D_\ep^{\lambda_1}(z) &= V_\ep (z) - V_\ep^{\lambda_1}(z) > 0 &\text{for } \lambda_1 < |z| \le R_1,\\
\quad V_\ep^{\lambda_1}(z) &\le (1-2\eta_2)\, \alpha_{N,s}\, |z|^{-(N-2s)} &\text{for } |z| \ge R_1
\end{aligned}
\end{equation}
and
\begin{equation}\label{eq-u_k}
\int_{B^N(0,R_1)} V_\ep^{p-\ep}(x,0)\, dx \ge \(1-{\eta_2 \over 2}\) \int_{\mr^N} w^p(x)\, dx
\end{equation}
for any sufficiently small $\ep > 0$.

Furthermore, we also have
\begin{equation}\label{eq-u_k-2}
V_\ep (z) \ge (1-\eta_2)\, \alpha_{N,s}\, |z|^{-(N-2s)} \quad \text{for } R_1 \le |z| \le \delta_1M_\ep^{2 \over N-2s}
\end{equation}
if $\delta_1 > 0$ is small enough. To verify it, let us choose a function $\hat{v}_{\ep}$ which solves
\[(-\Delta)^s \hat{v}_\ep = V_{\ep}^{p-\ep}(\cdot,0) \quad \text{in } B^N \Big(0, r_0 M_\ep^{2 \over N-2s}\Big) \quad \text{and} \quad \hat{v}_\ep = 0 \quad \text{on } \pa B^N \Big(0, r_0 M_\ep^{2 \over N-2s}\Big),\]
and denote by $\widehat{V}_\ep$ its $s$-harmonic extension to the cylinder $B^N (0, r_0 M_\ep^{\frac{2}{N-2s}}) \times (0,\infty)$.
Then the comparison principle \cite[Lemma 2.1]{CKL} tells us that $V_\ep \ge \widehat{V}_\ep$.
Since $H_\mc^1(z,y)$ is bounded in $\{(z,y) \in \mr^{N+1}_+ \times \mr^N: |z|,\, |y| \le 1/2\}$, we obtain
\begin{equation}\label{eq-u_k-4}
H_\mc^1 ((x,t),y) \le {\eta_2 \over 4} \cdot
{\gamma_{N,s} \over |(x-y,t)|^{N-2s}} \quad \text{for } |(x,t)|,\, |y| \le {\delta_1 \over r_0}
\end{equation}
by making $\delta_1 \in (0, r_0)$ smaller if necessary.
Moreover, because
\[|(x-y,t)| \le \(1-{1 \over l}\)|(x,t)| \quad \text{for }|(x,t)| \ge lR_1 \text{ and } |y| \le R_1\]
given any large $l > 1$, we see from \eqref{eq-G-R}, \eqref{eq-u_k} and \eqref{eq-u_k-4} that
\begin{equation}\label{eq-wve}
\begin{aligned}
\widehat{V}_\ep(x,t) &= \int_{B^N\big(0, r_0M_{\ep}^{2 \over N-2s}\big)} V_{\ep}^{p-\ep}(y,0)\, G_{\mc}^{r_0 M_\ep^{2 \over N-2s}} ((x,t),y)\, dy
\\
&\ge \(1-{\eta_2 \over 4}\) \int_{B^N\big(0, \delta_1 M_\ep^{2 \over N-2s}\big)} V_{\ep}^{p-\ep}(y,0)\, {\gamma_{N,s} \over |(x-y,t)|^{N-2s}}dy \\
&\ge \(1-{\eta_2 \over 2}\) \(\int_{B^N(0,R_1)} V_\ep^{p-\ep}(y,0)\, dy\) {\gamma_{N,s}  \over |(x,t)|^{N-2s}}
\\
&\ge \(1-\eta_2\) \(\int_{\mr^N} w^p(y)\, dy\) {\gamma_{N,s} \over |(x,t)|^{N-2s}}\\
&= \(1-\eta_2\) {\alpha_{N,s} \over |(x,t)|^{N-2s}} \qquad \text{for } lR_1 \le |(x,t)| \le \delta_1 M_\ep^{2 \over N-2s}
\end{aligned}
\end{equation}
by choosing $l$ large enough.
If $R_1 \le |z| \le lR_1$, we have $V_\ep (z) \ge (1-\eta_2)\, \alpha_{N,s}\, |z|^{-(N-2s)}$ for $\ep > 0$ small, for $V_\ep$ converges to $W$ uniformly over a compact set.
This shows the validity of \eqref{eq-u_k-2}.

\medskip \noindent \textsc{Step 3.} Suppose that \eqref{eq-inf-bound} does not hold with $\delta_1 > 0$ chosen in the previous step. Then
\[\min_{\{z \in \mr^{N+1}_+: |z|=r_k\}} V_{\ep_k} (z) > (1+\zeta_2) W(r_k, 0)\]
for some sequences $\{\ep_k\}_{k \in \mn}$ and $\{r_k\}_{k \in \mn}$ of positive numbers such that $\ep_k \to 0$ and $r_k \in (0, \delta_1 M_{\ep_k}^{2 \over N-2s})$.
Because of \eqref{eq-V_ep-W}, it should hold that $r_k \to \infty$.
Thus Lemma \ref{lem-app} implies
\begin{equation}\label{eq-contrary}
\min_{\{z \in \mr^{N+1}_+: |z|=r_k\}} V_k (z) \ge \(1+{\zeta_2 \over 2}\) \alpha_{N,s} r_k^{-(N-2s)}
\end{equation}
where $V_k := V_{\ep_k}$.

\medskip
Now we employ the method of moving spheres to the function $D_k^{\lambda}$ (see Remark \ref{rem-moving} (2) for its definition).
For any $k \in \mn$ and $\mu \in [\lambda_1, \lambda_2]$, let
\[\Sigma^{\mu}_k = \left\{x \in \overline{\mr^{N+1}_{+}}: \mu < |z| < r_k \right\}\]
and define a number $\bar{\lambda}_k$ by
\[\bar{\lambda}_k = \sup \left\{ \lambda \in [\lambda_1, \lambda_2] : D_k^{\mu}(z) \ge 0 \text{ in } \Sigma_k^{\mu} \text{ for all } \lambda_1 \le \mu \le \lambda \right\}.\]
By \eqref{eq-w_k} and \eqref{eq-u_k-2}, we see that $\bar{\lambda}_k \ge \lambda_1$.
We shall show that $\bar{\lambda}_k = \lambda_2$ for sufficiently large $k \in \mn$.

To the contrary, assume that $\bar{\lambda}_k < \lambda_2$ for some large fixed index $k \in \mn$.
By continuity it holds that $D_k^{\bar{\lambda}_k} \geq 0$ in $\Sigma_k^{\bar{\lambda}_k}$.
Moreover, from \eqref{eq-contrary} and \eqref{eq-u_k-3}, we have $D_k^{\bar{\lambda}_k} > 0$ on $\{z \in \mr^{N+1}_+: |z| = r_k\}$,
which implies that $D_k^{\bar{\lambda}_k} \neq 0$ in $\Sigma_k^{\bar{\lambda}_k}$.
Thus it holds that $D_{k}^{\bar{\lambda}_k} > 0$ in $\Sigma_k^{\bar{\lambda}_k}$ thanks to the strong maximum principle.
Pick $\delta > 0$ small so that the maximum principle for domains with small volume \cite[Lemma 2.8]{FW} can be applied.
If we choose a compact set $K \subset \Sigma_k^{\bar{\lambda}_k}$ such that $|\Sigma_k^{\bar{\lambda}_k} \setminus K| < \delta$, then $\inf_{K} D_{k}^{\bar{\lambda}_k} >0$, and then by continuity again, for $\lambda \in (\bar{\lambda}_k, \lambda_2)$ sufficiently close to $\bar{\lambda}_k$ we have
\[K \subset \Sigma_k^{\lambda}, \quad |\Sigma_k^{\lambda} \setminus K| < \delta \quad \text{and} \quad \inf_{K} D_k^{\lambda} > 0.\]
Then we see from \cite[Lemma 2.8]{FW} that $D_{k}^{\lambda} \geq 0$, contradicting the maximality of $\bar{\lambda}_k$.
Consequently, it should hold that $\bar{\lambda} = \lambda_2$.

\medskip
Finally, taking a limit $k \to \infty$ to $D_k^{\lambda_2} \ge 0$ in $\Sigma_k^{\lambda_2}$, we get
\[W(z) \ge W^{\lambda_2} (z) \quad \text{in } |z| \ge \lambda_2.\]
However it is impossible since $\lambda_2 > 1$.
Therefore \eqref{eq-inf-bound} should be true.
\end{proof}

We now complete the proof of Proposition \ref{prop-moving}.
\begin{lem}
Let $\{V_{\ep}\}_{\ep > 0}$ be the family in the statement of Proposition \ref{prop-moving} and $\delta_1 > 0$ the number selected in the proof of the previous lemma.
Then there exist a constant $C > 0$ and small parameter $\delta_0 \in (0,\delta_1)$ such that
\[V_{\ep} (z) \le C W(z) \quad \text{for all } z \in B^{N+1}_+\Big(0, \delta_0 M_{\ep}^{\frac{2}{N-2s}}\Big)\]
provided that $\ep > 0$ is sufficiently small.
\end{lem}
\begin{proof}
By virtue of Lemmas \ref{lem-moser-1} and \ref{lem-app}, we have a point $z_0 = (x_0, t_0) \in \mr^{N+1}_+$ such that $|z_0|=\delta_2 M_{\ep}^{\frac{2}{N-2s}}$ and
\[V_{\ep}(z_0) \leq (1 + \zeta_2) W (|z_0|,0) \le (1 + 2\zeta_2)\, \alpha_{N,s}\, |z_0|^{-(N-2s)}\]
for any small $\delta_2 \in (0, \delta_1)$.
Let $G^*_{\mc}$ be Green's function of \eqref{eq-Green} in the semi-infinite cylinder $\mc = B^N(0,\delta_1 M_{\ep}^{\frac{2}{N-2s}}) \times (0,\infty)$ (refer to Remark \ref{rem-moving} (3)).
Then we are able to choose a constant $\delta_3 \in (0, \delta_2)$ so small that
\begin{align*}
V_{\ep}(z_0) & \ge \int_{B^N(0,\delta_1 M_{\ep}^{\frac{2}{N-2s}})} V_{\ep}^{p-{\ep}}(y,0)\, G^*(z_0, y)\, dy\\
& \ge (1 - \zeta_2)\gamma_{N,s} \int_{B^N(0,\delta_2 M_{\ep}^{\frac{2}{N-2s}})} V_{\ep}^{p-{\ep}}(y,0)\, \frac{1}{|(x_0-y,t_0)|^{N-2s}}\, dy \\
& \ge (1 - 2\zeta_2) \gamma_{n,s}\, |z_0|^{-(N-2s)} \int_{B^N(0,\delta_3 M_{\ep}^{\frac{2}{N-2s}})} V_{\ep}^{p-{\ep}}(y,0)\, dy
\end{align*}
as in \eqref{eq-wve}.
Combining the above two estimates with \eqref{eq-u_k}, we obtain
\begin{equation}\label{eq-U-har-2}
\int_{B^N(0,\delta_3 M_{\ep}^{\frac{2}{N-2s}})\setminus B^N(0,R_1)} V_{\ep}^{p-\ep} (y,0)\, dy \leq C\zeta_2.
\end{equation}
Since $V_{\ep}$ is uniformly bounded, we observe from \eqref{eq-U-har-2} that
\begin{equation}\label{eq-V-est}
\int_{B^N(0,\delta_3 M_{\ep}^{\frac{2}{N-2s}})\setminus B^N(0,R_1)} V_{\ep}^{p+1} (y,0)\, dy \leq C \zeta_2.
\end{equation}
Now let us define $V_{r,\ep}(z) = r^{\frac{N-2s}{2}} V_{\ep}(rz)$ on the half-annulus $\{z \in \mr^{N+1}_+: 1/2 \le |z| \le 2\}$ for each $2R_1 \le r \le \delta_3 M_{\ep}^{\frac{2}{N-2s}}/2$ and $\ep > 0$ small.
Then one can apply the Moser iteration method with \eqref{eq-V-est} (refer to \cite{CKL}) to deduce that it is uniformly bounded in $\{z \in \mr^{N+1}_+: 3/4 \le |z| \le 3/2\}$, $r$ and $\ep$.
As a result, the Harnack inequality \cite[Lemma 4.9]{CS2} yields
\[\sup_{\{z \in \mr^{N+1}_+: 3/4 \le |z| \le 3/2\}} V_{r,\ep}(z) \leq C \inf_{\{z \in \mr^{N+1}_+: 3/4 \le |z| \le 3/2\}} V_{r,\ep}(z)\]
where $C > 0$ is a universal constant.
This inequality with Lemma \ref{lem-moser-1} and \eqref{eq-V_ep-W} concludes the proof of the lemma (giving $\delta_0 = 3\delta_3/4$).
\end{proof}

The following assertion is an immediate consequence of Proposition \ref{prop-moving}.

\begin{cor}\label{cor-ms}
Fix any $x_0 \in \mr^N$ and small $r_0 > 0$.
Let $\{U_{\ep}\}_{\ep >0}$ be a family of positive solutions to
\[\begin{cases}
\textnormal{div}(t^{1-2s} \nabla U_{\ep})=0 &\textnormal{in } B^N(x_0,r_0) \times (0,\infty), \\
\pa_{\nu}^s U_{\ep} = U_{\ep}^{p-\ep} &\textnormal{on } B^N(x_0,r_0),\\
\|U_{\ep}\|_{L^{\infty}(B^{N+1}_+((x_0,0),r_0))} \leq c M_{\ep}^{\frac{N-2s}{2}}
\end{cases}\]
for a certain constant $c > 0$ independent of $\ep$
and a family of positive values $\{M_{\ep}\}_{\ep >0}$ such that $\lim_{\ep \to \infty} M_{\ep} = \infty$ and $\lim_{\ep \to \infty} M_{\ep}^{\ep}=1$.
Suppose that the rescaled function $M_{\ep}^{-\frac{N-2s}{2}} U_{\ep}(M_{\ep}^{-1}\cdot + (x_0,0))$ converges weakly to the function $W$ in $\drn$. Then we have
\[U_{\ep}(z) \leq C M_{\ep}^{\frac{N-2s}{2}} W(M_{\ep}(z-(x_0,0))) \quad \text{for all } z \in B^{N+1}_+((x_0,0),\delta_0)\]
for some $\delta_0 \in (0,r_0)$ and $C > 0$ independent of $\ep$.
\end{cor}

\section{Application of the Pointwise Upper Estimate} \label{sec-point}
In this section, we gather refined information on finite energy solutions $U_{\ep}$ to equation \eqref{eq-ext}.
More precisely, we first show that $V_0$ vanishes identically if $m \neq 0$ in \eqref{eq-cc}.
Then we prove that any two different blow-up points do not collide and blow-up rates of each bubbles are compatible to the others.
Finally, we get sharp pointwise upper bounds of $U_{\ep}$ over the whole cylinder $\mc$,
and deduce that a suitable $L^{\infty}$-normalization of $U_{\ep}$ converges to a certain function as $\ep \searrow 0$, which can be described as a combination of Green's function.

\medskip
Recall from \eqref{eq-cc}, \eqref{eq-lam-order} and \eqref{eq-x_i} that
\begin{equation}\label{eq-cc-2}
U_n = V_0 + \sum_{i=1}^{m} PW_{\lambda_n^i, x_n^i} + R_n \quad \text{in } \mc'
\end{equation}
and $x_0^i = \lim_{n \to \infty} x_n^i \in \Omega$ for each $i = 1, \cdots, m$.
We also remind with \eqref{eq-cc-cond} that the concentration rate $\lambda_n^i$ on each blow-up part tends to 0 as $n \to \infty$.
The next lemma ensures that this convergence is not too fast.
\begin{lem}\label{lem-apriori}
Let $\{U_n\}_{n \in \mn}$ be a sequence of solutions to \eqref{eq-ext} with $\ep = \ep_n \searrow 0$, which admits a decomposition of the form \eqref{eq-cc-2}.
Then we have $\lim_{n \to \infty} (\lambda_n^i)^{\ep_n} = 1$ for each $1 \leq i \leq m$.
\end{lem}
\begin{proof}
Fix any $i \in \{1, \cdots, m\}$. Multiplying \eqref{eq-ext} by $PW_{\lambda_n^i, x_n^i}$, integrating by parts and using \eqref{wlyxt}, we get the equality
\begin{equation}\label{eq-ap-1}
\int_{\Omega} u_n^{p-\ep_n} Pw_{\lambda_n^i, x_n^i} dx = \kappa_s \int_{\mc} t^{1-2s} \nabla  U_n \cdot \nabla PW_{\lambda_n^i, x_n^i} dxdt = \int_{\Omega} u_n w_{\lambda_n^i, x_n^i}^p dx.
\end{equation}
Let us estimate the leftmost and rightmost sides of \eqref{eq-ap-1}.
By making use of \eqref{eq-cc-2}, \eqref{eq-cc-cond}, the mean value theorem, and the fact that $v_0$ is bounded on $\Omega \times \{0\}$ and $\lim_{n \to \infty} \|R_n\|_{\hc} = 0$, we obtain
\begin{align*}
&\ \int_{\Omega} \left| \( u_n^{p-\ep_n} - (Pw_{\lambda_n^i, x_n^i})^{p-\ep_n} \) Pw_{\lambda_n^i, x_n^i}\right| dx \\
&\leq C \int_{\Omega} \left| \sum_{j \neq i} Pw_{\lambda_n^j,x_n^j} + v_0 + r_n \right| \( \sum_{j=1}^m (Pw_{\lambda_n^j,x_n^j})^{p-1-\ep_n} + |v_0|^{p-1-\ep_n} + |r_n|^{p-1-\ep_n}\) Pw_{\lambda_n^i, x_n^i} dx = o(1).
\end{align*}
Hence it holds
\begin{equation}\label{eq-ap-7}
\int_{\Omega} u_n^{p-\ep_n} Pw_{\lambda_n^i, x_n^i} dx = \int_{\Omega}( Pw_{\lambda_n^i, x_n^i})^{p+1-\ep_n} dx + o(1).
\end{equation}
Moreover, it is easy to check that
\begin{equation}
\begin{aligned}\label{eq-ap-3}
\int_{\Omega} (Pw_{\lambda_n^i, x_n^i})^{p+1-\ep_n} dx
&= (\lambda_n^i)^{-\(\frac{N-2s}{2}\) \ep_n} \int_{\lambda_n^i (\Omega-x_n^i)} (Pw_{1,0})^{p+1-\ep_n} dx \\
&= (\lambda_n^i)^{-\(\frac{N-2s}{2}\) \ep_n} \( \int_{\mr^N} w^{p+1} dx + o(1)\).
\end{aligned}
\end{equation}
Similarly, one may show that
\begin{equation}\label{eq-ap-4}
\int_{\Omega} u_n w_{\lambda_n^i, x_n^i}^p dx = \int_{\mr^N} w^{p+1} dx + o(1).
\end{equation}
Inserting \eqref{eq-ap-7}, \eqref{eq-ap-3} and \eqref{eq-ap-4} into \eqref{eq-ap-1}, we conclude that $\lim_{n \to \infty} ( \lambda_n^i)^{\ep_n} =1$. The lemma is proved.
\end{proof}

In the following, we give the proof of several claims stated in the beginning of this section, applying the previous lemma.
\begin{lem}\label{lem-alt}
Let $\{U_n\}_{n \in \mn}$ be a sequence of solutions of \eqref{eq-ext} with $\ep = \ep_n$ which admits an asymptotic behavior \eqref{eq-cc-2}.
Suppose that there exists at least one bubble in \eqref{eq-cc-2}, i.e., $m \ne 0$.
Then $V_0 \equiv 0$.
\end{lem}
\begin{proof}
Firstly, we aim to show that
\begin{equation}\label{eq-peak-0}
U_n(z) \le C (\lambda_n^1)^{-{N-2s \over 2}} \quad \text{uniformly for any } z \in \mc \text{ and } n \in \mn.
\end{equation}
To do so, we consider the function $\wtu_n (z) := (\lambda_n^1)^{\frac{N-2s}{2}} U_n (\lambda_n^1 z)$ defined in $\mc_n: = (\lambda_n^1)^{-1} \mc$.
One can easily observe that it satisfies
\[\begin{cases}
\text{div}(t^{1-2s} \nabla \wtu_n) = 0 &\text{in } \mc_n, \\
\wtu_n = 0 &\text{on } \pa_L \mc_n,
\\
\pa_{\nu}^s \wtu_n = (\lambda_n^1)^{\frac{(N-2s)\ep}{2}} \wtu_n^{p-\ep} &\text{on } \Omega_n \times \{0\}
\end{cases}\]
where $\Omega_n := (\lambda_n^1)^{-1} \Omega$. Also it is plain to check
\begin{equation}\label{eq-peak-1}
\sup_{n \in \mn} \int_{\mathcal{C}_n} t^{1-2s}|\nabla \wtu_n (x,t)|^2 dx dt < C\quad\textrm{and}\quad \sup_{n \in \mn}\int_{\Omega_n} |\wtu_n (x,0)|^{\frac{2N}{N-2s}} dx < C.
\end{equation}
Owing to H\"older's inequality, it holds that
\[\sup_{n \in \mn} \int_{B^N (y,r_0) \cap \Omega_n} |\wtu_n (x,0)|^2 dx < C\]
for any $y \in \Omega_n$ and a small value $r_0 > 0$ to be fixed soon.
Combining this with the first estimate of \eqref{eq-peak-1} yields
\begin{equation}\label{eq-peak-2}
\sup_{n \in \mn} \int_{B^{N+1}_+((y,0),r_0) \cap \mc_n} t^{1-2s}|\wtu_n (x,t)|^2 dx dt < C
\end{equation}
(see the proof of \cite[Lemma 3.1]{CK}).
Let $\delta > 0$ be the number in Lemma \ref{lem-high-2}.
Then, from \eqref{eq-lam-order}, \eqref{eq-cc-2} and the fact that \[\lim_{n \to \infty} \int_{\Omega_n} \left|(\lambda_n^1)^{\frac{N-2s}{2}} R_n (\lambda_n^1 x,0)\right|^{\frac{2N}{N-2s}} dx = 0,\]
it is possible to choose $r_0 >0$ small enough so that
\[\sup_{n \in \mn} \int_{B^N (y,r_0) \cap \Omega_n} |\wtu_n (x,0)|^{\frac{2N}{N-2s}} dx < \delta.\]
Therefore, by invoking Lemma \ref{lem-high-2} with $a= (\lambda_n^1)^{\frac{(N-2s)\ep}{2}} \wtu_n^{p-1-\ep}$ and $f=0$, we may conclude that
\[\sup_{n \in \mn} \|\wtu_n\|_{L^{\infty}(B^{N}(y,{r_0}/{2}) \cap \Omega_n)} \leq C \sup_{n \in \mn} \int_{B^{N+1}_+((y,0),r_0) \cap \mc_n} t^{1-2s} |\wtu_n (x,t)|^2 dxdt \leq C\]
where the last inequality is due to \eqref{eq-peak-2}. Since $y \in \Omega_n$ is chosen arbitrarily and $\wtu_n$ attains its maximum on $\Omega_n \times \{0\}$, it follows
\[\sup_{n \in \mn} \sup_{(x,t) \in \mc_n} \wtu_n (x,t) = \sup_{n \in \mn} \sup_{x \in \Omega_n} \wtu_n (x,0) \leq C.\]
This proves \eqref{eq-peak-0}.

Now, by virtue of \eqref{eq-peak-0}, Corollary \ref{cor-ms} and Lemma \ref{lem-apriori}, we obtain \begin{equation}\label{eq-peak-split-1}
U_n(z) \le C(\lambda_n^1)^{-\frac{N-2s}{2}} W\({z-(x_n^1,0) \over \lambda_n^1}\) \quad \text{for all } z \in B^{N+1}_+((x_n^1,0), \delta_0),
\end{equation}
which implies
\[\lim_{n \to \infty} U_n(z) = 0 \quad \text{for any } z \in B^{N+1}_+((x_0^1,0), \delta_0/2) \setminus \{(x_0^1,0)\}.\]
Since $R_n(\cdot, 0) \to 0$ in $L^{2N \over N-2s}(\Omega)$, there exists a point $x' \in B^N (x_0^1, \delta_0/2) \setminus \{x_0^1, \cdots, x_0^m\}$ such that $\lim_{n \to \infty} R_n(x',0) = 0$.
Furthermore, we know from \eqref{eq-cc-2} that $U_n(x,0) \ge V_0(x,0) + R_n(x,0)$ for all $x \in \Omega$, so it should hold that $V_0 (x',0) = 0$.

On the other hand, each $U_n$ and its weak limit $V_0$ are nonnegative in $\mc$.
Therefore one concludes from the strong maximum principle that $V_0 \equiv 0$.
\end{proof}

In Lemmas \ref{lem-peak-split}-\ref{lem-u-asym}, we are mainly interested on the case $m \ne 0$.
In this case, solutions $U_n$ to \eqref{eq-ext} with the asymptotic behavior \eqref{eq-cc-2} can be rewritten in the form
\begin{equation}\label{eq-refined}
U_n = \sum_{i=1}^{m} PW_{\lambda_n^i, x_n^i} + R_n \quad \text{in } \mc'
\end{equation}
where $\lim_{n \to \infty} \|R_n\|_{\hc} = 0$.

\begin{lem}\label{lem-peak-split}
Assume that a sequence $\{U_n\}_{n \in \mn}$ of solutions to \eqref{eq-ext} with $\ep = \ep_n$ has the asymptotic behavior given by Lemma \ref{lem-cc-bounded} with $m \ge 1$.
Then there exists a constant $d_0 > 0$ such that
\begin{equation}\label{eq-peak-split}
|x_0^i - x_0^j| \ge d_0 \quad \text{for any } 1 \le i < j \le m.
\end{equation}
\end{lem}
\begin{proof}
Assume that two different blow-up points converge to the same point $x' \in \Omega$.
By \eqref{eq-cc-cond} and \eqref{eq-lam-order}, one of the following holds: \[\text{(1) } \lim_{n \to \infty} \frac{\lambda_n^i}{\lambda_n^j} = 0
\quad \text{or} \quad \text{(2) } \lim_{n \to \infty} \frac{|x_n^i - x_n^j|^2}{\lambda_n^i \lambda_n^j} = \infty.\]

Suppose that (1) holds. Then by \eqref{eq-lam-order} it should be true that \begin{equation}\label{eq-peak-split-6}
\lim_{n \to \infty} {\lambda_n^1 \over \lambda_n^m} = 0.
\end{equation}
We shall prove that it cannot happen.
By Corollary \ref{cor-ms}, we have an upper bound \eqref{eq-peak-split-1}. Furthermore, we can find a lower bound
\begin{equation}\label{eq-peak-split-2}
U_n(z) \geq C (\lambda_n^m)^{\frac{N-2s}{2}} \quad \text{for all } z \in B^{N+1}_+((x',0),\delta_0)
\end{equation}
where $\delta_0 > 0$ is a number in \eqref{eq-peak-split-1} (taken smaller if required).
Indeed, by \eqref{eq-pw-exp}, \eqref{eq-cc}, \eqref{eq-cc-cond} and Lemma \ref{lem-alt}, we have
\[(\lambda_n^m)^{\frac{N-2s}{2}} u_n \(\lambda_n^m y + x_n^m\) \to w(y) \quad \text{for a.e. } y \in \mr^N.\]
Thus Green's representation formula, Fatou's lemma and Lemma \ref{lem-apriori} show
\begin{equation}\label{eq-peak-split-4}
\begin{aligned}
 U_n(z) &\geq  \int_{B^N(x_n^m,\delta_0)} G_{\mc}(z,x)\, u_n^{p-\ep_n}(x)\, dx
\geq C  \int_{B^N(x_n^m,\delta_0)} u_n^{p-\ep_n}(x)\, dx
\\
& = C (\lambda_n^m)^{\frac{N-2s}{2}(1+\ep_n)} \int_{B^N(0, \delta_0/\lambda_n^m)}
 \left[ (\lambda_n^m)^{\frac{N-2s}{2}} u_n \( \lambda_n^my + x_n^m\) \right]^{p-\ep_n} dy
\\
&\geq C \(\int_{\mr^N} w^p (y)\, dy + o(1) \) (\lambda_n^m)^{\frac{N-2s}{2}},
\end{aligned}
\end{equation}
which confirms \eqref{eq-peak-split-2}.
Now fixing any point $z^* \in \mr^{N+1}_+$ such that $|z^*-(x',0)| = \delta_0/2$ and putting it into \eqref{eq-peak-split-1} and \eqref{eq-peak-split-2},
we discover that $(\lambda_n^m)^{N-2s \over 2} \le C (\lambda_n^1)^{N-2s \over 2}$ for some $C > 0$, contradicting \eqref{eq-peak-split-6}.
Therefore (1) is false and we may assume that
\begin{equation}\label{eq-peak-split-5}
\lim_{n \to \infty} \frac{\lambda_n^i}{\lambda_n^j} = c_0 \quad \text{for some } c_0 \in (0,1].
\end{equation}

Assume that (2) is true. Owing to \eqref{eq-peak-split-5}, inequality \eqref{eq-peak-0} can be written as
\begin{equation}\label{eq-peak-split-7}
U_n(z) \le C (\lambda_n^j)^{-{N-2s \over 2}} \le C (\lambda_n^i)^{-{N-2s \over 2}} \quad \text{for } z \in \mc \text{ and } n \in \mn.
\end{equation}
Hence we infer from elliptic regularity and Corollary \ref{cor-ms} that
\[(\lambda_n^j)^{\frac{N-2s}{2}} u_n \(\lambda_n^j \cdot + x_n^j\) \to w \quad \text{in } C^{\alpha}(\mr^N) \text{ for some } \alpha \in (0,1)\]
and
\begin{equation}\label{eq-peak-split-3}
U_n(z) \le C(\lambda_n^i)^{-\frac{N-2s}{2}} W\({z-(x_n^j,0) \over \lambda_n^i} + {(x_n^j-x_n^i,0) \over \lambda_n^i}\)
\end{equation}
for all $z \in B^{N+1}_+((x',0), \delta_0/2)$ and large $n \in \mn$.
Since $\lim_{n \to \infty} |x_n^j - x_n^i|/\lambda_n^i = \infty$ holds because of \eqref{eq-lam-order},
if we take $z = (x_n^j, 0)$ in inequality \eqref{eq-peak-split-3} and use \eqref{eq-peak-split-7}, then we get
\[C (\lambda_n^j)^{-{N-2s \over 2}} \le u_n(x_n^j) \le C (\lambda_n^i)^{-{N-2s \over 2}} w\({x_n^j-x_n^i \over \lambda_n^i}\) = o(1) \cdot (\lambda_n^i)^{-{N-2s \over 2}}\]
provided $n \in \mn$ large.
However, this is absurd as \eqref{eq-peak-split-5} holds, and so (2) does not hold either.

Summing up, every possible case is excluded if two blow-up points tend to the same point.
Accordingly, \eqref{eq-peak-split} has the validity.
\end{proof}

In the following lemma, we study the behavior of solutions $u_n$ to \eqref{eq-main} outside the blow-up points $\{x_0^1, \cdots, x_0^m\}$. We set
\begin{equation}\label{eq-A_r}
A_r = \Omega \setminus \bigcup_{i=1}^m B^N(x_0^i, r) \quad \text{for any } r > 0.
\end{equation}
\begin{lem}\label{lem-u-zero}
Suppose that $\{U_n\}_{n \in \mn}$ is a family of solutions for \eqref{eq-ext} with $\ep = \ep_n$ satisfying the asymptotic behavior \eqref{eq-refined}.
Then for any small $r > 0$, we have $u_n(x) = O((\lambda_n^m)^{\frac{N-2s}{2}})$ uniformly for $x \in A_r$.
\end{lem}
\begin{proof}
Let $a_n = u_n^{p-1-\ep_n}$ so that $\pa_{\nu}^s U_n = a_nu_n$ in $\Omega \times \{0\}$.
Then we see from \eqref{eq-asym} that
\[\|a_n\|_{L^{N \over 2s}(A_{r/4})} \le C \(\sum_{i=1}^m \left\|w_{\lambda_n^i ,x_n^i}\right\|^{p-1-\ep_n}_{L^{p+1-\({N \over 2s}\)\ep_n} \(\mr^N \setminus B^N(x_0^i,{r/4})\)}
+ \|R_n\|_{\hc}\) = o(1).\]
Therefore we can proceed the Moser iteration argument to get $\|a_n\|_{L^q(A_{r/2})} = o(1)$ for some $q > {N \over 2s}$, and it further leads to $\|u_n\|_{L^{\infty}(A_r)} = o(1)$ (see Section 3 in \cite{CKL}).

Assume that $r \in (0, \min\{\delta_0, d_0/2\})$ where $\delta_0 > 0$ and $d_0$ are the numbers picked up in Corollary \ref{cor-ms} and Lemma \ref{lem-peak-split}, respectively.
Then the argument used to derive \eqref{eq-peak-0} with Lemma \ref{lem-peak-split} deduces
\[U_n(x,t) \le C(\lambda_n^i)^{-{N-2s \over 2}} \quad \text{for } |x - x_0^i| \le r \text{ and } t \ge 0\]
so that Corollary \ref{cor-ms} implies
\[u_n(x) \le C(\lambda_n^i)^{-\frac{N-2s}{2}} w\({x-x_n^i \over \lambda_n^i}\) \le C(\lambda_n^i)^{\frac{N-2s}{2}} \quad \text{for } {r \over 2} \le |x - x_0^i| \le r\]
where $i = 1, \cdots, m$.
By Green's representation formula, one may write
\[u_n(x) = \int_{A_{r/2}} G(x,y)\, u_n^{p-\ep_n}(y)\, dy + \sum_{i=1}^m \int_{B^N(x_0^i,r/2)} G(x,y)\, u_n^{p-\ep_n}(y)\, dy.\]
If we set $b_n = \|u_n\|_{L^{\infty}(A_r)}$, then we observe with assumption \eqref{eq-lam-order} that
\begin{equation}\label{eq-bound-1}
\begin{aligned}
\int_{A_{r/2}} G(x,y)\, u_n^{p-\ep_n}(y)\, dy
&\le C \int_{A_{r/2}} G(x,y) \(b_n^{p-\ep_n} + \max\{\lambda_n^1, \cdots, \lambda_n^m\}^{{\frac{N-2s}{2}}(p-\ep_n)}\) dy \\
&\le C \(b_n^{p-\ep_n} + \(\lambda_n^m\)^{{\frac{N-2s}{2}}(p-\ep_n)}\)
\end{aligned}
\end{equation}
for any $x \in A_r$. Besides, Corollary \ref{cor-ms} and Lemma \ref{lem-apriori} give us that
\begin{equation}\label{eq-bound-2}
\begin{aligned}
\int_{B^N(x_0^i,r/2)} G(x,y)\, u_n^{p-\ep_n}(y)\, dy &\le C \int _{B^N(x_0^i,r/2)} u_n^{p-\ep_n} (y)\, dy\\
&\le C \int_{B^N(x_0^i,r/2)} w_{\lambda_n^i,x_n^i}^{p-\ep_n}(y)\, dy
\le C (\lambda_n^i)^{\frac{N-2s}{2}}
\end{aligned}
\end{equation}
for all $x \in A_r$ and each $i = 1, \cdots, m$.
Hence, by combining \eqref{eq-bound-1} and \eqref{eq-bound-2}, we get
\[b_n \le C \(b_n^{p-\ep_n} + (\lambda_n^m)^{\frac{N-2s}{2}}\).\]
Since we have $p -\ep_n > 1$ and $b_n = o(1)$, the above inequality implies that $b_n \le C(\lambda_n^m)^{\frac{N-2s}{2}}$.
The lemma is proved.
\end{proof}

We prove the compatibility of the blow-up rates $\{\lambda_n^1, \cdots, \lambda_n^m\}$.
\begin{lem}\label{lem-com}
There exists a constant $C_0 > 0$ independent of $n \in \mn$ such that
\[\frac{\lambda_n^i}{\lambda_n^j} \le C_0 \quad \text{for any } 1 \leq i, j \leq m.\]
\end{lem}
\begin{proof}
As in \eqref{eq-peak-split-4}, it can be verified that $u_n (x) \geq C (\lambda_n^i)^{\frac{N-2s}{2}}$ in $\bigcup_{k=1}^m B^N(x_0^k, r)$ for each $i = 1, \cdots, m$.
As a matter of fact, it is possible to substitute $x_n^m$ and $\lambda_n^m$ in \eqref{eq-peak-split-4} with $x_n^i$ and $\lambda_n^i$, respectively.

On the other hand, we know from Lemma \ref{lem-u-zero} that $u_n (x) \leq C (\lambda_n^j)^{\frac{N-2s}{2}}$ for $x \in B^N(x_0^j,r) \setminus B^N(x_0^j, r/2)$.
Thus we have $(\lambda_n^i)^{\frac{N-2s}{2}} \leq C (\lambda_n^j)^{\frac{N-2s}{2}}$ for any $1 \leq i, j \leq m$.
The proof is done.
\end{proof}

As in the statement of Theorem \ref{thm-ch}, we set $b_i =\lim_{n \to \infty} \(\frac{\lambda_n^i}{\lambda_n^1}\)^{\frac{N-2s}{2}} \in (0, \infty)$ for any $i = 1, \cdots, m$.
\begin{lem}\label{lem-u-asym}
Suppose that $\{U_n\}_{n \in \mn}$ is a sequence of solutions to equation \eqref{eq-ext} with $\ep = \ep_n$ which admit the asymptotic behavior \eqref{eq-refined}. Then it holds
\begin{equation}\label{eq-u-ep-asym}
\lim_{n \to \infty} (\lambda_n^1)^{-\frac{N-2s}{2}} U_n(x,t) = c_1 \sum_{i=1}^m b_i\, G_{\mc}((x,t),x_0^i)
\end{equation}
in $C^0 (\mc' \setminus\{(x_0^1,0), \cdots, (x_0^m,0)\})$.
Furthermore, we have
\begin{equation}\label{eq-u-ep-asym-2}
\lim_{n \to \infty} (\lambda_n^1)^{-\frac{N-2s}{2}} \nabla_x^k U_n (x,t) = c_1 \sum_{i=1}^m b_i\, \nabla_x^k G_{\mc}((x,t),x_0^i)
\end{equation}
for $1 \le k \le 2$ and
\begin{equation}\label{eq-u-ep-asym-3}
\lim_{n \to \infty} (\lambda_n^1)^{-\frac{N-2s}{2}} t^{l-2s} \pa_t^l \nabla_x^k U_n (x,t) = c_1 \sum_{i=1}^m b_i\, t^{l-2s} \pa_t^l \nabla_x^k G_{\mc}((x,t),x_0^i)
\end{equation}
for any pair $(k, l)$ such that $0 \le k \le 1$, $1 \le l \le 2$ and $1 \le k+l \le 2$ in $C^0 (\mc' \setminus\{(x_0^1,0), \cdots, (x_0^m,0)\})$.
We remind that $\mc' = \Omega \times [0, \infty)$ and $c_1 = \int_{\mr^N} w^p(x)\, dx > 0$.
\end{lem}
\begin{proof}
Take any $r > 0$ small for which Lemma \ref{lem-u-zero} holds. We are concerned with the values of $U_n (z)$ for $z \in A_r' := \mc' \setminus \cup_{i=1}^m \overline{B^{N+1}_+((x_0^i,0), r)}$. Let us look at
\begin{equation}\label{eq-u-asym-1}
U_n(z) = \int_{A_{r/2}} G_{\mc}(z,y)\, u_n^{p-\ep_n}(y)\, dy + \sum_{i=1}^m \int_{B^N(x_0^i,r/2)} G_{\mc}(z,y)\, u_n^{p-\ep_n}(y)\, dy.
\end{equation}
Then by the previous lemma we have
\[(\lambda_n^1)^{-\frac{N-2s}{2}} \int_{A_{r/2}} G_{\mc}(z,y)\, u_n^{p-\ep_n}(y)\, dy
\le C (\lambda_n^1)^{-\frac{N-2s}{2}} (\lambda_n^m)^{\frac{N-2s}{2}{(p-\ep_n)}} \int_{\Omega} G_{\mc}(z,y)\, dy = o(1).\]
Let us decompose
\begin{align*}
&\int_{B^N(x_0^i,r/2)} G_{\mc}(z,y)\, u_n^{p-\ep_n}(y)\, dy\\
&= G_{\mc}(z,x_0^i) \int_{B^N(x_0^i,r/2)} u_n^{p-\ep_n}(y)\, dy + \int_{B^N(x_0^i,r/2)} (G_{\mc}(z,y) - G_{\mc}(z,x_0^i))\, u_n^{p-\ep_n}(y)\, dy
\end{align*}
for each $i \in \{1, \cdots, m\}$. Since
\[(\lambda_n^i)^{\frac{N-2s}{2}} u_n \( \lambda_n^i y + x_n^i\) \rightharpoonup w(y) \quad \text{weakly in } H^s(\mr^N),\]
according to Corollary \ref{cor-ms} and the Lebesgue dominated convergence theorem, we get
\[(\lambda_n^1)^{-\frac{N-2s}{2}} \int_{B^N(x_0^i,r/2)} u_n^{p-\ep_n}(y)\, dy \to b_i \int_{\mr^N} w^p (y)\, dy.\]
Also, employing the mean value theorem, we calculate
\begin{align*}
&\ \left|(\lambda_n^1)^{-\frac{N-2s}{2}} \int_{B^N(x_0^i,r/2)} (G_{\mc}(z,y)-G_{\mc}(z,x_0^i))\, u_n^{p-\ep_n}(y)\, dy \right|
\\
&\le (\lambda_n^1)^{-\frac{N-2s}{2}} \int_{B^N(x_0^i,r/2)} \sup_{z \in A'_r,\, a \in (0,1)} \left\|\nabla_y G_{\mc}(z,ay + (1-a)x_0^i)\right\| \cdot |y-x_0^i|\, u_n^{p-\ep_n}(y)\, dy
\\
&\le C (\lambda_n^1)^{-\frac{N-2s}{2}} r^{1-s} \int_{B^N(x_n^i,3r/4)} |y-x_0^i|^s\, u_n^{p-\ep_n}(y)\, dy
\\
&\le C b_i r^{1-s} \left[(\lambda_n^i)^s \(\int_{\mr^N} |y|^s w^p(y) \, dy + o(1)\) + |x_n^i-x_0^i|^s \(\int_{\mr^N} w^p(y) \, dy + o(1)\) \right] = o(1).
\end{align*}
Therefore, combining all the computations, we see that \eqref{eq-u-ep-asym} holds uniformly for $z = (x,t) \in A'_r$.
Since $r > 0$ is arbitrary, it follows that \eqref{eq-u-ep-asym} is valid in $C^0 (\mc' \setminus\{(x_0^1,0), \cdots, (x_0^m,0)\})$.

\medskip
In order to show \eqref{eq-u-ep-asym-2} and \eqref{eq-u-ep-asym-3}, we need some results on elliptic regularity. The proof is deferred to Appendix \ref{sec-app-b}.
\end{proof}
\begin{rem}
For the future use, we rewrite \eqref{eq-u-ep-asym} as
\begin{equation}\label{eq-lim-u}
\lim_{n \to \infty} (\lambda_n^1)^{-\frac{N-2s}{2}} U_n(x,t) = {c_3 b_i \over |(x-x_0^i,t)|^{N-2s}} + \mt_i(x,t)
\end{equation}
for $(x,t) \in \mc' \setminus\{(x_0^1,0), \cdots, (x_0^m,0)\}$ and $1 \le i \le m$.
Here $c_3 := c_1 \gamma_{N,s} > 0$ and $\mt_i$ is a map defined by
\begin{equation}\label{eq-mh-2}
\mt_i(x,t) = - c_1 b_i H_{\mc}((x,t),x_0^i) + c_1 \sum_{k \ne i} b_k G_{\mc}((x,t),x_0^k).
\end{equation}
If $r \in (0, d_0/2)$ where $d_0 > 0$ is set in Lemma \ref{lem-peak-split},
then \eqref{eq-prop-G} and \eqref{eq-prop-H} imply that
the functions $\mt_i$, ${\pa \mt_i \over \pa x_j}$ and $z \cdot \nabla \mt_i$ are $s$-harmonic in $B^{N+1}_+((x_0^i,0),r)$ for all $1 \le i \le m$ and $1 \le j \le N$, i.e.,
\begin{equation}\label{eq-mh-1}
\begin{cases}
\text{div}(t^{1-2s} \nabla \mt_i) = \text{div}\(t^{1-2s} \nabla \(\dfrac{\pa \mt_i}{\pa x_j}\)\) = \text{div}\(t^{1-2s} \nabla (z \cdot \nabla \mt_i)\) =  0 &\text{in } B^{N+1}_+((x_0^i,0),r),\\
\pa_{\nu}^s \mt_i = \pa_{\nu}^s \(\dfrac{\pa \mt_i}{\pa x_j}\) = \pa_{\nu}^s (z \cdot \nabla \mt_i) = 0 &\text{on } B^N(x_0^i,r)
\end{cases}
\end{equation}
holds.
\end{rem}

\section{Proof of Main Theorems for the Spectral Fractional Laplacians} \label{sec-bi}
This section is devoted to the proof of our main theorems.
To get the desired results, we will derive two identities regarding blow-up points and rates by exploiting a type of Green's identity.
For notational simplicity, we use $z-x_0^i$ to denote $(x-x_0^i,t)$ throughout the section.

\medskip
As before, let $\{U_n\}_{n \in \mn}$ be a sequence of solutions to \eqref{eq-ext} with $\ep = \ep_n$ of the form \eqref{eq-refined}.
We remind from \eqref{eq-ext} that $U_n$ is a solution of the problem
\begin{equation}\label{eq-u-b}
\begin{cases}
\text{div}(t^{1-2s} \nabla U_n)=0 &\text{in } \mc,\\
\pa_{\nu}^s U_n = U_n^{p-\ep_n} &\text{on } \Omega \times \{0\}.
\end{cases}
\end{equation}
By the translation and scaling invariance of \eqref{eq-u-b}, the functions $V= \frac{\pa U_n}{\pa x_j}$ and $V = (z - x_0^i) \cdot \nabla  U_n + \( \frac{2s}{p-1-\ep_n}\) U_n$ (for each $1 \le i \le m$ and $1 \le j \le N$) satisfy the equation
\begin{equation}\label{eq-v}
\begin{cases}
\text{div}(t^{1-2s} \nabla V)=0 &\text{in } \mc,
\\
\pa_{\nu}^s V = (p-\ep_n)U_n^{p-1-\ep_n} V &\text{on } \Omega \times \{0\}.
\end{cases}
\end{equation}
\begin{lem}\label{lem-uv}
Assume that a function $V \in \hc$ satisfies \eqref{eq-v}. Then for any point $y \in \Omega$, the following identity
\begin{equation}\label{eq-uv}
\kappa_s \int_{\pa_I B_+^{N+1}((y,0),r)} t^{1-2s} \( \frac{\pa U_n}{\pa \nu} V - \frac{\pa V}{\pa \nu} U_n \) dS_z
= (p-1-\ep_n) \int_{B^N (y,r)} U_n^{p-\ep_n} V \,dx
\end{equation}
holds for any $r \in (0, \textnormal{dist}(y, \pa \Omega))$.
\end{lem}
\begin{proof}
Multiplying the first equation of \eqref{eq-u-b} by $V$ and that of \eqref{eq-v} by $U_n$, and then integrating the results over $B_+^{N+1}((y,0),r)$, we obtain
\begin{align*}
\kappa_s \int_{\pa_I B_+^{N+1}((y,0),r)} t^{1-2s} \( \frac{\pa U_n}{\pa \nu} V - \frac{\pa V}{\pa \nu} U_n \) dS_z
&= - \int_{B^N(y,r)} (\pa_{\nu}^s U_n \cdot V - \pa_{\nu}^s V \cdot U_n)\, dx \\
&= (p-1-\ep_n) \int_{B^N(y,r)} U_n^{p-\ep_n} V \,dx.
\end{align*}
Here the second equality comes from the second equations of \eqref{eq-u-b} and \eqref{eq-v}.
This proves \eqref{eq-uv}.
\end{proof}

Based on the previous identity, we now deduce two kinds of information on the concentration points and rates.
\begin{lem}\label{lem-pt}
For any $1 \le i \le m$ and $1 \le j \le N$, we have $\frac{\pa \mt_i}{\pa x_j}(x_0^i,0)=0$ for $\mathcal{H}_i$ defined in \eqref{eq-mh-2}, or equivalently,
\begin{equation}\label{eq-con}
b_i \frac{\pa H}{\pa x_j} (x_0^i, x_0^i) - \sum_{k \ne i} b_k \frac{\pa G}{\pa x_j} (x_0^i, x_0^k) = 0.
\end{equation}
\end{lem}
\begin{proof}
Fix any $i \in \{1, \cdots, m\}$. Taking $V = \frac{\pa U_n}{\pa x_j}$ and $y = x_0^i$ in \eqref{eq-uv}, we have
\begin{equation}\label{eq-bi-7}
\begin{aligned}
&\kappa_s \int_{\pa_I B_+^{N+1}((x_0^i,0),r)} t^{1-2s} \left[ \frac{\pa U_n}{\pa \nu} \frac{\pa U_n}{\pa x_j} - \frac{\pa }{\pa \nu}\(\frac{\pa U_n}{\pa x_j}\) U_n \right] dS_z\\
&= (p-1-\ep_n) \int_{B^N (x_0^i,r)} U_n^{p-\ep_n} \frac{\pa U_n}{\pa x_j} \,dx
= \({p-1-\ep_n \over p+1-\ep_n}\) \int_{\pa B^N (x_0^i,r)} U_n^{p+1-\ep_n} \nu_j\, dS_x.
\end{aligned}
\end{equation}
By Lemmas \ref{lem-apriori}, \ref{lem-u-zero} and \ref{lem-com},
\begin{equation}\label{eq-bi-5}
(\lambda_n^1)^{-(N-2s)} \left| \int_{\pa B^N (x_0^i,r)} U_n^{p+1-\ep_n} \nu_j\, dS_x \right| = (\lambda_n^1)^{-(N-2s)} O ((\lambda_n^i)^{N-{N-2s \over 2}\ep_n}) = o(1).
\end{equation}
Hence we see from \eqref{eq-bi-7} and \eqref{eq-bi-5} that
\begin{equation}\label{eq-bi-8}
\lim_{n \to \infty} (\lambda_n^1)^{-(N-2s)} \int_{\pa_I B_+^{N+1}((x_0^i,0),r)} t^{1-2s} \left[ \frac{\pa U_n}{\pa \nu} \frac{\pa U_n}{\pa x_j}
- \frac{\pa }{\pa \nu}\(\frac{\pa U_n}{\pa x_j}\) U_n \right] dS_z = 0.
\end{equation}
Using \eqref{eq-lim-u}, we evaluate the left-hand side of \eqref{eq-bi-8} as follows:
\begin{align*}
&\lim_{n \to \infty} (\lambda_n^1)^{-(N-2s)} \int_{\pa_I B_+^{N+1}((x_0^i,0),r)} t^{1-2s} \left[ \frac{\pa U_n}{\pa \nu} \frac{\pa U_n}{\pa x_j}
- \frac{\pa}{\pa \nu}\(\frac{\pa U_n}{\pa x_j}\) U_n \right] dS_z
\\
&= \int_{\pa_I B_+^{N+1}((x_0^i,0),r)} t^{1-2s} \({(N-2s)c_3 b_i \over |z-x_0^i|^{N-2s+1}} - {\pa \mt_i \over \pa \nu}(z)\)
\cdot \({(N-2s) c_3 b_i (x-x_0^i)_j \over |z-x_0^i|^{N-2s+2}} - {\pa \mt_i \over \pa x_j}(z)\)
\\
&\hspace{65pt} + t^{1-2s} \frac{\pa}{\pa \nu} \({(N-2s)c_3 b_i (x-x_0^i)_j \over |z-x_0^i|^{N-2s+2}} - {\pa \mt_i \over \pa x_j}(z)\) \cdot \({c_3 b_i \over |z-x_0^i|^{N-2s}} + \mt_i(z)\) dS_z
\\
&= \int_{\pa_I B_+^{N+1}((x_0^i,0),r)} t^{1-2s} \left[-{(N-2s)c_3 b_i \over |z-x_0^i|^{N-2s+1}}{\pa \mt_i \over \pa x_j}(z)
- {(N-2s) c_3 b_i (x-x_0^i)_j \over |z-x_0^i|^{N-2s+2}} {\pa \mt_i \over \pa \nu}(z) \right.
\\
&\hspace{100pt} \left. + \frac{\pa}{\pa \nu} \({(N-2s)c_3 b_i (x-x_0^i)_j \over |z-x_0^i|^{N-2s+2}}\) \mt_i(z)\right] dS_z
\\
&\ - \int_{\pa_I B_+^{N+1}((x_0^i,0),r)} t^{1-2s} \left[ \frac{\pa}{\pa \nu}\({\pa \mt_i \over \pa x_j}\) {c_3 b_i \over |z-x_0^i|^{N-2s}} \right] dS_z
\\
&\ + \int_{\pa_I B_+^{N+1}((x_0^i,0),r)} t^{1-2s} \left[ {\pa \mt_i \over \pa \nu} {\pa \mt_i \over \pa x_j} - \frac{\pa}{\pa \nu}\({\pa \mt_i \over \pa x_j}\) \mt_i \right] dS_z
\\
&:= I_1 + I_2 + I_3.
\end{align*}
Let us compute each of the terms $I_1,\, I_2$ and $I_3$.
Firstly, \eqref{eq-mh-1} yields that
\begin{equation}\label{eq-bi-2}
I_3 = - \int_{B^N(x_0^i,r)} \left[ \pa_{\nu}^s \mt_i \cdot \({\pa \mt_i \over \pa x_j}\) - \pa_{\nu}^s\({\pa \mt_i \over \pa x_j}\) \cdot \mt_i \right] dx = 0.
\end{equation}
Also, according to estimates \eqref{eq-prop-h1} and \eqref{eq-prop-h2}, we have
\begin{equation}\label{eq-bi-6}
\begin{aligned}
\lim_{r \to 0} |I_2| &\le \lim_{r\to 0} \int_{\pa_I B_+^{N+1}((x_0^i,0),r)} t^{1-2s} \left| \frac{\pa}{\pa \nu}\({\pa \mt_i \over \pa x_j}\) {c_3 b_i \over |z-x_0^i|^{N-2s}} \right| dS_z
\\
&\leq C \lim_{r \to 0} \int_{\pa_I B_+^{N+1}((x_0^i,0),r)} \frac{(t^{1-2s} +1)}{|z-x_0^i|^{N-2s}} dS_z
\le C \lim_{r \to 0} (r + r^{2s}) = 0.
\end{aligned}
\end{equation}
Therefore we only need to compute $\lim_{r \to 0} I_1$. By homogeneity, its first term is calculated to be
\begin{align*}
&- \lim_{r \to 0} \int_{\pa_I B_+^{N+1}((x_0^i,0),r)} t^{1-2s} {(N-2s)c_3 b_i \over |z-x_0^i|^{N-2s+1}}{\pa \mt_i \over \pa x_j}(z)\, dS_z
\\
&= - {\pa \mt_i \over \pa x_j}(x_0^i, 0) \cdot (N-2s)c_3 b_i \int_{\pa_I B_+^{N+1}(0,1)} \frac{t^{1-2s}}{|z|^{N-2s+1}} dS_z.
\end{align*}
For the second term, one can deduce
\begin{align*}
&- \lim_{r \to 0} \int_{\pa_I B_+^{N+1}((x_0^i,0),r)} t^{1-2s} {(N-2s) c_3 b_i (x-x_0^i)_j \over |z-x_0^i|^{N-2s+2}} {\pa \mt_i \over \pa \nu}(z)\, dS_z
\\
& = - (N-2s) c_3 b_i \cdot \lim_{r \to 0} \int_{\pa_I B_+^{N+1}((x_0^i,0),r)} \sum_{k=1}^{N+1} \frac{t^{1-2s} (x-x_0^i)_j (x-x_0^i)_k}{|z-x_0^i|^{N-2s+3}} {\pa \mt_i \over \pa x_k}(z)\, dS_z
\\
&= - {\pa \mt_i \over \pa x_j}(x_0^i, 0) \cdot (N-2s) c_3 b_i \int_{\pa_I B_+^{N+1}(0,1)} \frac{t^{1-2s} x_j^2}{|z|^{N-2s+3}} \, dS_z,
\end{align*}
because the mean value formula with \eqref{eq-prop-h1} and \eqref{eq-prop-h2} imply
\[\left|\frac{t^{1-2s} (x-x_0^i)_j (x-x_0^i)_k}{|z-x_0^i|^{N-2s+3}} \({\pa \mt_i \over \pa x_k}(z) - {\pa \mt_i \over \pa x_k}(x_0^i, 0)\)\right| \leq C\frac{(1+ t^{1-2s}) |z-x_0^i|^{3}}{|z-x_0^i|^{N-2s+3}} = C \frac{1+t^{1-2s}}{|z-x_0^i|^{N-2s}}\]
for $1 \le j,\, k \le N+1$
so that the value of its integration over the half-sphere ${\pa_I B_+^{N+1}((x_0^i,0),r)}$ is bounded by $C(r+r^{2s})$ (see \eqref{eq-bi-6}).
Finally, by direct computation, we discover
\begin{align*}
&\lim_{r \to 0} \int_{\pa_I B_+^{N+1}((x_0^i,0),r)} t^{1-2s} \frac{\pa}{\pa \nu} {(N-2s)c_3 b_i (x-x_0^i)_j \over |z-x_0^i|^{N-2s+2}} \mt_i(z)\, dS_z
\\
&= -(N-2s) (N-2s+1) c_3 b_i \lim_{r \to 0} \int_{\pa_I B_+^{N+1}((x_0^i,0),r)} t^{1-2s} {(x-x_0^i)_j \over |z-x_0^i|^{N-2s+3}} \mt_i(z)\, dS_z
\\
&= - {\pa \mt_i \over \pa x_j}(x_0^i, 0) \cdot (N-2s) (N-2s+1) c_3 b_i \int_{\pa_I B_+^{N+1}(0,1)} \frac{t^{1-2s} x_j^2}{|z|^{N-2s+3}} \, dS_z
\end{align*}
where we used $\mt_i (x, 0) = \mt_i (x_0^i, 0) + (x-x_0^i) \cdot \nabla_x \mt_i (x_0^i, 0) + O(|x-x_0^i|^2)$ to find the second equality.
Thus \eqref{eq-bi-8} is reduced to
\[- {\pa \mt_i \over \pa x_j}(x_0^i, 0) \cdot \(\int_{\pa_I B_+^{N+1}(0,1)} \frac{t^{1-2s}}{|z|^{N-2s+1}} dS_z
+ (N-2s+2) \int_{\pa_I B_+^{N+1}(0,1)} \frac{t^{1-2s} x_j^2}{|z|^{N-2s+3}} \, dS_z\)  = 0.\]
Therefore ${\pa \mt_i \over \pa x_j}(x_0^i, 0) =0$, proving the lemma.
\end{proof}

\begin{rem}
It is shown in \cite[Section 4]{CKL} that
\begin{equation}\label{eq-beta}
\int_{\pa_I B_+^{N+1}(0,1)} \frac{t^{1-2s}}{|z|^{N-2s+1}} dS_z = {|S^{N-1}| \over 2} B\(1-s, {N \over 2}\)
\end{equation}
and
\[\int_{\pa_I B_+^{N+1}(0,1)} \frac{t^{1-2s} x_1^2}{|z|^{N-2s+3}} dS_z = {|S^{N-1}| \over 2N} B\(1-s, {N+2 \over 2}\) = {1 \over N-2s+2} \int_{\pa_I B_+^{N+1}(0,1)} \frac{t^{1-2s}}{|z|^{N-2s+1}} dS_z\]
where $B$ is the Beta function.
\end{rem}

\begin{lem}\label{lem-con2}
For each $1 \leq i \leq m$ we have
\begin{equation}\label{eq-con2}
b_i^2 H(x_0^i, x_0^i) - \sum_{k \neq i} b_i b_k G(x_0^i, x_0^k) = \frac{ c_2}{2c_1} b_0
\end{equation}
where $c_2 > 0$ in \eqref{eq-c12} and $b_0 = \lim_{n \to \infty} (\lambda_n^1)^{-(N-2s)} \ep_n$.
\end{lem}
\begin{proof}
Fix $i \in \{1, \cdots, m\}$.
Taking $V = V_n = (z-x_0^i) \cdot \nabla U_n + \( \frac{2s}{p-1-\ep_n} \) U_n$ and $y = x_0^i$ in \eqref{eq-uv}, we find
\begin{equation}\label{eq-bi-4}
\begin{aligned}
&\kappa_s \lim_{n \to \infty} (\lambda_n^1)^{-(N-2s)} \int_{\pa_I B_+^{N+1}((x_0^i,0),r)} t^{1-2s} \left[\frac{\pa U_n}{\pa \nu} V_n - \frac{\pa V_n}{\pa \nu} U_n \right] dS_z \\
&\qquad= \lim_{n \to \infty} (\lambda_n^1)^{-(N-2s)} (p-1-\ep_n) \int_{B^N (x_0^i,r)} u_n^{p-\ep_n} v_n\, dx
\end{aligned}
\end{equation}
where $v_n = \text{tr}|_{\Omega \times \{0\}} V_n$.
To evaluate the left-hand side of \eqref{eq-bi-4}, we observe from \eqref{eq-lim-u} that
\[\lim_{n \to \infty} (\lambda_n^1)^{-\frac{N-2s}{2}}V_n (z)
= - \(\frac{N-2s}{2}\) \frac{c_3 b_i}{|z-x_0^i|^{N-2s}} + (z-x_0^i) \cdot \nabla \mt_i (z) + \(\frac{N-2s}{2}\) \mt_i (z)\]
for $z = (x,t) \in \mc' \setminus\{(x_0^1,0), \cdots, (x_0^m,0)\}$.
Thus we get
\begin{align*}
& \lim_{n \to \infty} (\lambda_n^1)^{-(N-2s)} \int_{\pa_I B_+^{N+1}((x_0^i,0),r)} t^{1-2s} \left[ \frac{\pa U_n}{\pa \nu} V_n - \frac{\pa V_n}{\pa \nu} U_n \right] dS_z
\\
&= - \int_{\pa_I B_+^{N+1}((x_0^i,0),r)} t^{1-2s} \frac{(N-2s) c_3b_i}{|z-x_0^i|^{N-2s+1}} \( (z-x_0^i) \cdot \nabla \mt_i + (N-2s) \mt_i\) dS_z
\\
&\ - \int_{\pa_I B_+^{N+1}((x_0^i,0),r)} t^{1-2s} \frac{c_3 b_i }{|z-x_0^i|^{N-2s}} \frac{\pa}{\pa \nu} \((z-x_0^i) \cdot \nabla \mt_i + (N-2s) \mt_i\) dS_z
\\
&\ + \int_{\pa_I B_+^{N+1}((x_0^i,0),r)} t^{1-2s} \left[\frac{\pa \mt_i}{\pa \nu} \((z-x_0^i) \cdot \nabla \mt_i + \(\frac{N-2s}{2}\) \mt_i\) \right.
\\
&\hspace{100pt} - \left.\mt_i \frac{\pa}{\pa \nu} \((z-x_0^i) \cdot \nabla \mt_i + \(\frac{N-2s}{2}\) \mt_i\) \right] dS_z
\\
&:= J_1 + J_2 + J_3.
\end{align*}
As the previous proof, let us estimate each of $J_1,\, J_2$ and $J_3$.
As demonstrated in \eqref{eq-bi-2}, we have $J_3 =0$.
Besides \eqref{eq-prop-h1} and \eqref{eq-prop-h2} lead us to derive
\[\lim_{r \to 0} |J_2| \leq  C \lim_{r \to 0} \int_{\pa_I B_+^{N+1}((x_0^i,0),r)} \frac{(t^{1-2s}+1)}{|z-x_0^i|^{N-2s}} dS_z = 0.\]
Lastly, since
\[\left. \left[(z-x_0^i) \cdot \nabla \mt_i(z) + (N-2s) \mt_i(z) \right] \right|_{z=(x_0^i,0)} = (N-2s) \mt_i(x_0^i,0),\]
we have
\[\lim_{r \to 0} J_1 = -c_3b_i(N-2s)^2 \(\int_{\pa_I B_+^{N+1}(0,1)} \frac{t^{1-2s}}{|z|^{N-2s+1}} dS_z\) \mt_i (x_0^i, 0).\]
As a result, after the limit $r \to 0$ being taken, the left-hand side of \eqref{eq-bi-4} becomes
\begin{equation}\label{eq-bi-9}
c_1c_3\kappa_s (N-2s)^2 \(\int_{\pa_I B_+^{N+1}(0,1)} \frac{t^{1-2s}}{|z|^{N-2s+1}} dS_z\) \left[b_i^2 H_{\mc}((x_0^i,0),x_0^i) - \sum_{k \ne i} b_i b_k G_{\mc}((x_0^i,0),x_0^k)\right].
\end{equation}
Meanwhile, using integration by parts, we deduce that
\begin{align*}
&\int_{B^N (x_i^0,r)} u_n^{p-\ep_n} \left[ (x-x_0^i) \cdot \nabla_x u_n + \( \frac{2s}{p-1-\ep_n} \) u_n \right] dx
\\
&= {1 \over p+1-\ep_n} \int_{B^N (x_i^0,r)} (x-x_0^i) \cdot \nabla_x u_n^{p+1-\ep_n} \, dx
+ \frac{2s}{p-1-\ep_n} \int_{B^N (x_i^0,r)} u_n^{p+1-\ep_n} dx
\\
&= {1 \over p+1-\ep_n} \int_{\pa B^N (x_i^0,r)} (x-x_0^i) \cdot \nu u_n^{p+1-\ep_n} \, dS_x
+ \(\frac{2s}{p-1-\ep_n} - \frac{N}{p+1-\ep_n}\) \int_{B^N (x_i^0,r)} u_n^{p+1-\ep_n} dx.
\end{align*}
Note that
\[\frac{2s}{p-1-\ep_n} - \frac{N}{p+1-\ep_n} = {(N-2s) \ep_n \over \({4s \over N-2s} - \ep_n\) \({2N \over N-2s} - \ep_n\)} = {(N-2s)^3 \ep_n \over 8Ns} (1+o(1))\]
and
\[\int_{\pa B^N (x_i^0,r)} (x-x_0^i) \cdot \nu u_n^{p+1-\ep_n} \, dS_x = O \((\lambda_n^1)^N\).\]
Hence the right-hand side of \eqref{eq-bi-4} equals to
\begin{equation}\label{eq-bi-10}
(\lambda_n^1)^{-(N-2s)} \ep_n (1+o(1)) \cdot {(N-2s)^2 \over 2N} \int_{\mr^N} w^{p+1} dx + O\((\lambda_n^1)^{2s}\).
\end{equation}
From \eqref{eq-bi-4}, \eqref{eq-bi-9}, \eqref{eq-bi-10} and \eqref{eq-beta}, we get
\begin{align*}
{b_0 \over N} \int_{\mr^N} w^{p+1} dx
&= c_1c_3\kappa_s |S^{N-1}| B\(1-s, {N \over 2}\) \left[b_i^2 H_{\mc}((x_0^i,0),x_0^i) - \sum_{k \ne i} b_i b_k G_{\mc}((x_0^i,0),x_0^k)\right]\\
&= {2 \over N-2s} \(\int_{\mr^N} w^p dx\)^2 \left[b_i^2 H(x_0^i,x_0^i) - \sum_{k \ne i} b_i b_k G(x_0^i,x_0^k)\right].
\end{align*}
This completes the proof.
\end{proof}

\medskip
We are now prepared to complete the proof of our main theorems.

\begin{proof}[Proof of Theorem \ref{thm-ch}]
Assume that $\sup_{n \in \mn} \|u_n\|_\mh < \infty$.
Then, if we let $U_n$ be the $s$-harmonic extension of $u_n$ over the half-cylinder $\mc = \Omega \times (0,\infty)$, we have $\sup_{n \in \mn} \|U_n\|_{\hc} < \infty$ by inequality \eqref{eq-U_ep}.
Thus we can apply Lemma \ref{lem-cc-bounded} to the sequence $\{U_n\}_{n \in \mn}$ to deduce the existence of
an integer $m \in \mn \cup \{0\}$ and sequences of positive numbers and points $\{(\lambda_n^i, x_n^i)\}_{n \in \mn} \subset (0,\infty) \times \Omega$ for each $i = 1, \cdots, m$
such that relation \eqref{eq-cc-cond} holds (in particular $\lambda_n^i \to 0$) and
\begin{equation}\label{eq-m1-2}
U_n - \(V_0 + \sum_{i=1}^{m} PW_{\lambda_n^i, x_n^i}\) \to 0
\text{ in } \hc \quad \text{as } n \to \infty
\end{equation}
along a subsequence.
Here $V_0$ is the weak limit of $U_n$ in $\hc$, which is a solution to \eqref{eq-V-eq}, and $PW_{\lambda_n^i, x_n^i}$ is the projected bubble whose definition can be found in \eqref{eq-pw}.

We now split the problem into two cases.

\noindent \textbf{Case 1 ($m=0$).}
By \eqref{eq-CS} and the strong maximum principle, $v_0(x) = V_0(x,0)$ for $x \in \Omega$ satisfies equation \eqref{eq-l}.
In addition, by \eqref{eq-m1-2}, it holds that
\[\lim_{n \to \infty} \|u_n - v\|_{\mh} = \lim_{n \to \infty} \|U_n - V_0\|_{\hc} = 0.\]
This case corresponds to the first alternative $(1)$ of Theorem \ref{thm-ch}.

\noindent \textbf{Case 2 ($m \geq 1$).}
Thanks to Lemma \ref{lem-alt}, we have $V_0 = 0$ in this situation.
Hence \eqref{eq-m1-2} and discussion in Subsection \ref{subsec-cc} give decomposition \eqref{eq-asym} as well as $x_n^i \to x_0^i \in \Omega$.
Also, by Lemmas \ref{lem-peak-split} and \ref{lem-com}, there are constants $d_0,\ C_0 > 0$ independent of $n \in \mn$ such that
\[|x_0^i - x_0^j| \ge d_0 \quad \text{and} \quad \frac{\lambda_n^i}{\lambda_n^j} \le C_0 \quad \text{for any } 1 \le i \ne j \le m.\]
Thus we may set a positive value $b_i =\lim_{n \to \infty} \(\frac{\lambda_n^i}{\lambda_n^1}\)^{\frac{N-2s}{2}}$ for each $1 \leq i \leq m$.
Furthermore, Lemmas \ref{lem-pt} and \ref{lem-con2} imply that $((b_1, \cdots, b_m), (x_0^1, \cdots, x_0^m)) \subset (0,\infty)^m \times \Omega^m$
is a critical point of the function $\Phi_m$ introduced in \eqref{eq-phi}.
We have proved that the case $m \geq 1$ corresponds to the second alternative $(2)$ in Theorem \ref{thm-ch}. The proof is finished.
\end{proof}

\begin{proof}[Proof of Theorem \ref{thm-main2}]
The fact that $M$ is a nonnegative matrix can be shown as in Appendix A of \cite{BLR}, so we left it to the reader.

\medskip
Suppose that $M$ is nondegenerate.
Since the left-hand side of \eqref{eq-con2} is finite, it should hold that $b_0 \in [0, \infty)$.
To the contrary, let us assume that $b_0 = 0$.
Then we see
\[b_i H(x_0^i, x_0^i) - \sum_{k \neq i} b_j G(x_0^i, x_0^k) = 0\]
for each $1 \leq i \leq m$.
It means that $\textbf{b} = (b_1, \cdots, b_m)$ is a nonzero vector such that $M\textbf{b}=0$.
However this is nonsense because the nondegeneracy condition of $M$ tells us that $\textbf{b} = 0$.
Hence $b_0 \neq 0$ should be true, and thus
\[\lim_{n \to \infty} \log_{\ep_n} \lambda_n^i = \lim_{n \to \infty}  \log_{\ep_n} \left[\ep_n^{1 \over N-2s} \(b_0^{-{1 \over N-2s}} + o(1)\) \(b_i^{2 \over N-2s} + o(1)\)\right] = {1 \over N-2s}.\]
The proof is now complete.
\end{proof}

\section{The Restricted Fractional Laplacian and the Classical Laplacian} \label{sec-res}
\subsection{Proof of Theorems \ref{thm-ch} and \ref{thm-main2} for the Restricted Fractional Laplacian}
Here we briefly mention how the proof for the main theorems \ref{thm-ch} and \ref{thm-main2} can be carried out for the restricted fractional Laplacian.

\medskip
First of all, as mentioned before, the Struwe's concentration-compactness principle type result (\textbf{Step 1} in Introduction) can be obtained as in \cite{A, FG, PP}.
Besides the moving plain argument in Section \ref{sec-moving} (corresponding to \textbf{Step 2}) is local in nature, so the same proof as in Section \ref{sec-moving} works.
For Section \ref{sec-point}, one can check each lemma remains valid even if \eqref{eq-ext} is replaced with \eqref{eq-ext-2}.
Finally, we notice that Lemmas \ref{lem-pt} and \ref{lem-con2} were obtained from the information on the solutions $\{U_n\}_{n \in \mn}$ to \eqref{eq-ext} over the half-balls $\{B_+^{N+1}((x_0^i,0),r)\}_{i=1}^m$.
Therefore the same argument goes through for \eqref{eq-ext-2}, completing  \textbf{Step 3}.
Theorems \ref{thm-ch} and \ref{thm-main2} for the restricted fractional Laplacians now follow.

\subsection{Proof of Theorem \ref{thm-local}} \label{subsec-local}
To validate Theorem \ref{thm-local}, we follow the strategy used to prove Theorems \ref{thm-ch} and \ref{thm-main2} for nonlocal problems.

\medskip
The representation formula \eqref{eq-cc-local} of finite energy solutions $\{u_n\}_{n \in \mn}$ to \eqref{eq-local} is due to Struwe \cite{S} (\textbf{Step 1}).
Also, as in \cite[Appendix A]{CKL2}, a moving sphere argument can be applied to deduce a pointwise upper bound of $u_n$.
It implies Lemmas \ref{lem-alt}, \ref{lem-peak-split} and \ref{lem-com} for the local case, which are originally given in \cite{Sch}.
It can be easily seen that Lemma \ref{lem-apriori} remains true, and the local versions of Lemmas \ref{lem-u-zero} and \ref{lem-u-asym} are found in \cite[Section 2]{CKL2}, whence \textbf{Step 2} is finished.
Regarding Lemma \ref{eq-uv}, we have
\begin{lem}
Suppose that a function $v \in H^{1,2}_0(\Omega)$ satisfies
\[- \Delta v = (p-\ep_n)\, u_n^{p-1-\ep_n} v \quad \text{in } \Omega.\]
Then for any point $y \in \Omega$, the following identity
\begin{equation}\label{eq-uv-local}
\int_{\pa B^N(y,r)} \({\pa u \over \pa \nu}v - {\pa v \over \pa \nu}u\) dS_x = (p-1-\ep_n) \int_{B^N(y,r)} u_n^{p-\ep_n} v\, dx
\end{equation}
holds for any $r \in (0, \textnormal{dist}(y,\pa\Omega))$.
\end{lem}
\noindent By taking $u = u_n$ and $v = {\pa u_n \over \pa x_j}$ for $j = 1, \cdots, N$ or $v = (x-x_0^i) \cdot \nabla u_n + \({2 \over p-1-\ep_n}\) u_n$ for $i = 1, \cdots, m$ in \eqref{eq-uv-local},
we get Lemmas \ref{lem-pt} and \ref{lem-con2} where the constants $c_1$ and $c_2$ are given by \eqref{eq-c12} with $s = 1$.
Thus \textbf{Step 3} is done.
Putting all the results together, we complete the proof of Theorem \ref{thm-local}.

\appendix
\section{Lower and Upper Estimates of the Standard Bubble in $\mr^{N+1}_+$} \label{sec-W}
Here we shall prove a decay estimate of $W_{\lambda,0}$, which is necessary in applying the moving sphere argument (see Section \ref{sec-moving}).
\begin{lem}\label{lem-app}
Then for any $\eta > 0$ there exists $R = R(\eta) > 1$ so large that
\begin{equation} \label{eq-app}
\alpha_{N,s} (1-\eta) \lambda^{N-2s \over 2} |z|^{-(N-2s)}  \leq W_{\lambda,0} (z) \leq \alpha_{N,s} (1+\eta) \lambda^{N-2s \over 2} |z|^{-(N-2s)} \quad \text{for all } |z| > R
\end{equation}
where $\alpha_{N,s} > 0$ is the constant defined in Notations.
\end{lem}
\begin{proof}
Since $W_{\lambda,0}(z) = \lambda^{-{N-2s \over 2}} W(\lambda^{-1}z)$, we may assume that $\lambda = 1$.
Let us prove the lower estimate first.
Taking a small number $\delta > 0$ to be determined later, we consider two exclusive cases: (1) $|x| > \delta |t|$ and (2) $|x| \leq \delta |t|$.

For the case (1), we see from Green's representation formula, \eqref{eq-Green-R} and \eqref{eq-bubble} that
\begin{equation}\label{eq-U-sharp-1}
\begin{aligned}
W(x,t) &\geq \alpha_{N,s}^p \gamma_{N,s} \int_{|y| \leq \delta |x|} \frac{1}{|(x-y,t)|^{N-2s}} \frac{1}{(1+|y|^2)^{N+2s \over 2}} dy
\\
&\geq \frac{1}{|((1+\delta)x,t)|^{N-2s}} \cdot \alpha_{N,s}^p \gamma_{N,s} \int_{|y| \leq \delta |x|} \frac{1}{(1+|y|^2)^{N+2s \over 2}} dy
\\
&\geq \frac{1}{(1+\delta)^{N-2s}|(x,t)|^{N-2s}}\(\alpha_{N,s}^p \gamma_{N,s} \int_{\mr^N} \frac{1}{(1+|y|^2)^{\frac{N+2s}{2}}} dy - o(1)\)
\\
&= \frac{1}{(1+\delta)^{N-2s}|(x,t)|^{N-2s}}\(\alpha_{N,s} - o(1)\)
\end{aligned}
\end{equation}
where $o(1) \to 0$ as $|z| = |(x,t)| \to \infty$.

For the case (2), we have
\begin{equation}\label{eq-U-sharp-2}
\begin{aligned}
W(x,t) & \ge \alpha_{N,s}^p \gamma_{N,s} \int_{|y| \leq \delta |t|} \frac{1}{|(x-y,t)|^{N-2s}} \frac{1}{(1+|y|^2)^{N+2s \over 2}} dy
\\
& \ge \frac{1}{(1+2\delta)^{N-2s}|t|^{N-2s}} \cdot  \alpha_{N,s}^p \gamma_{N,s} \int_{|y| \leq \delta |t|} \frac{1}{(1+|y|^2)^{N+2s \over 2}} dy
\\
& \ge \frac{1}{(1+2\delta)^{N-2s}|(x,t)|^{N-2s}} (\alpha_{N,s} - o(1))
\end{aligned}
\end{equation}
where $o(1) \to 0$ as $|z| = |(x,t)| \to \infty$.

Hence if we choose $\delta > 0$ small and $R > 0$ large so that
\[{1 \over (1+2\delta)^{N-2s}} \ge 1-{\eta \over 2} \quad \text{and} \quad \alpha_{N,s} - o(1) \ge \(1-{\eta \over 2}\) \alpha_{N,s},\]
we obtain the desired estimate from \eqref{eq-U-sharp-1} and \eqref{eq-U-sharp-2}.

\medskip
We turn to prove the upper estimate. Again, we take into account the cases (1) $|x| > \delta |t|$ and (2) $|x| \leq \delta |t|$ separately.

For the case (1), we estimate
\[\alpha_{N,s}^p \gamma_{N,s} \int_{|y| \leq \delta |x|} \frac{1}{|(x-y,t)|^{N-2s}} \frac{1}{(1+|y|^2)^{N+2s \over 2}} dy
\leq \frac{\alpha_{N,s}}{|((1-\delta) x, t)|^{N-2s}} \leq \frac{1}{(1-\delta)^{N-2s}} \frac{\alpha_{N,s}}{|(x,t)|^{N-2s}}\]
and
\begin{align*}
&\ \alpha_{N,s}^p \gamma_{N,s} \int_{|y| \geq \delta |x|} \frac{1}{|(x-y,t)|^{N-2s}} \frac{1}{(1+|y|^2)^{N+2s \over 2}} dy \\
&= \alpha_{N,s}^p \gamma_{N,s} \(\int_{2|x| \geq |y| \geq \delta |x|} + \int_{|y|\geq 2|x|}\) \frac{1}{|(x-y,t)|^{N-2s}} \frac{1}{(1+|y|^2)^{N+2s \over 2}} dy
\\
& \leq \alpha_{N,s}^p \gamma_{N,s} \( \int_{2|x| \geq |y| \geq \delta |x|} \frac{1}{|x-y|^{N-2s}} \frac{1}{(\delta |x|)^{N+2s}} dy
+ \int_{|y| \geq 2|x|} \frac{1}{|x|^{N-2s}} \frac{1}{|(1+|y|^2)^{\frac{N+2s}{2}}} dy \)
\\
&\leq \frac{\alpha_{N,s}'}{\delta^{N+2s}|x|^N} \leq \frac{2^{N/2} \alpha_{N,s}'}{\delta^{2(N+s)}|(x,t)|^N},
\end{align*}
where $\alpha_{N,s}' > 0$ is a certain constant relying only on $N$ and $s$.
Observe that the last inequality came from $|(x,t)| < \sqrt{1+\delta^{-2}} |x| \le \sqrt{2} \delta^{-1}|x|$ for $\delta > 0$ small enough.
Combining the above estimates, we get
\begin{equation}\label{eq-U-sharp-3}
W(x,t) \leq \frac{1}{(1-\delta)^{N-2s}} \frac{\alpha_{N,s}}{|(x,t)|^{N-2s}} + \frac{2^{N/2} \alpha_{N,s}'}{\delta^{2(N+s)}|(x,t)|^N}.
\end{equation}
For the case (2), we have
\begin{equation}\label{eq-U-sharp-4}
W(x,t) \leq \alpha_{N,s}^p \gamma_{N,s} \int_{\mr^N} \frac{1}{|t|^{N-2s}} \frac{1}{(1+|y|^2)^{N+2s \over 2}} dy  = \frac{\alpha_{N,s}}{|t|^{N-2s}} \leq (1+\delta)^{N-2s} \frac{\alpha_{N,s}}{|(x,t)|^{N-2s}}.
\end{equation}
Consequently, with the choices
\[{1 \over (1-\delta)^{N-2s}} \le 1 + {\eta \over 2} \quad \text{and} \quad \frac{2^{N/2} \alpha_{N,s}'}{\delta^{2(N+s)}R^N} \le {\eta \over 2},\]
estimates \eqref{eq-U-sharp-3} and \eqref{eq-U-sharp-4} imply the second inequality of the lemma. The proof is completed.
\end{proof}

\section{Elliptic Regularity Results and Derivation of \eqref{eq-u-ep-asym-2} and \eqref{eq-u-ep-asym-3}} \label{sec-app-b}
This section is devoted to present some elliptic regularity results and its application to justification of \eqref{eq-u-ep-asym-2} and \eqref{eq-u-ep-asym-3}.
For brevity, we denote
\[Q_r = B^{N+1}_+((x,0), r) \quad \text{and} \quad B_r = B^{N}(x,r) \quad \text{for any fixed } x \in \Omega, 0 < r < \text{dist}(x, \pa \Omega)/2.\]
Also $\pa_i = \pa_{x_i}$ for $1 \le i \le N$.

\medskip
We need to recall two lemmas which can be proved with Moser's iteration method.
One is an a priori $L^{\infty}$-estimate. See e.g. \cite[Lemma 3.8]{CKL}, \cite[Theorem 3.4]{GQ} and \cite[Propositions 2.3, 2.6]{JLX}.
\begin{lem}\label{lem-high-2}
Let $U \in H^{1,2}_0(Q_{2r};t^{1-2s})$ be a weak solution to
\[\begin{cases}
\textnormal{div}(t^{1-2s} \nabla U) =0 &\text{in } Q_{2r},\\
\pa_{\nu}^s U = a U + f &\text{on } B_{2r}
\end{cases}\]
and assume that $\|a\|_{L^{\frac{N}{2s}}(B_{2r})} < \delta$ for a small value $\delta = \delta (N,s)>0$.
If $f \in L^q(B_r)$ for some $q > {n \over 2s}$ and $\theta \in (0,1)$, then we have
\[\| U\|_{L^{\infty}(Q_{\theta r})}^2 + \int_{Q_{\theta r}} t^{1-2s} |\nabla  U|^2 dz \leq C \( \int_{Q_r} t^{1-2s} |U|^2 dz + \|f\|_{L^q (B_r)}^2\)\]
for some $C = C(N,s,r,\theta) > 0$.
\end{lem}
\noindent The other is a result on H\"older estimates. Refer to \cite[Proposition 2.6]{JLX} and \cite[Lemma 4.5]{CS2}.
\begin{lem}\label{lem-high-1}
Let $U \in H^{1,2}_0(Q_{2r};t^{1-2s})$ be a weak solution to
\[\begin{cases}
\textnormal{div}(t^{1-2s} \nabla U) =0 &\text{in } Q_{2r},\\
\pa_{\nu}^s U = f &\text{on } B_{2r},
\end{cases}\]
and $\theta \in (0,1)$.

\noindent (1) If $f \in L^q (B_r)$ for some $q > \frac{n}{2s}$, then for some $\alpha \in (0,1)$ we have
\[\|U\|_{C^{\alpha}(Q_{\theta r})} \leq C\(\| U\|_{L^{\infty}(Q_r)} + \|f\|_{L^q (B_r)}\).\]

\noindent (2) If $f \in C^{\beta} (B_r)$ for some $\beta \in (0,1)$, then there exists $\alpha \in (0,1)$ such that
\[\| t^{1-2s} \pa_t U \|_{C^{\alpha}(Q_{\theta r})} \leq C \(\| U\|_{L^{\infty}(Q_r)} + \|f\|_{C^{\beta}(B_r)}\).\]
\end{lem}

Now we are ready to prove the main result of this section.
\begin{prop}\label{prop-high}
Let $1 < q \leq \frac{N+2s}{N-2s}$. Suppose that $U \in H^{1,2}_0(Q_{2r};t^{1-2s})$ is a positive solution of
\begin{equation}\label{eq-high-2}
\begin{cases}
\textnormal{div}(t^{1-2s} \nabla U) =0 &\text{in } Q_{2r},\\
\pa_{\nu}^s U = U^q &\text{on } B_{2r}.
\end{cases}
\end{equation}
Assume that $\int_{B_{2r}} U^{\frac{N}{2s}(q-1)^2} (x,0)\, dx \leq \delta$ for some small value $\delta = \delta (N,s) >0$.
Then $U(x,t)$ is twice differentiable in the $x$-variable in $Q_{r/2}$.
Moreover, the following estimates hold:
\begin{align*}
\| \nabla_x U \|_{C^{\alpha} (Q_{r/2})} &\leq C \(1+ \|U^{q-1}\|_{L^{\infty}(B_r)}\) \|U\|_{L^{\infty}(Q_r)},\\
\| t^{1-2s} \pa_t U\|_{C^{\alpha}(Q_{r/2})} &\leq C \(\|U\|_{L^{\infty}(Q_r)} + \|U^q\|_{C^1 (B_r)}\),\\
\|\nabla_x^2 U\|_{C^{\alpha} (Q_{r/2})} &\leq C
\(1+ \|U^{q-1}\|_{L^{\infty}(B_r)}\)
\(\|U\|_{L^{\infty}(Q_r)} + \| U^{q-2} |\nabla_x U|^2\|_{L^{\infty}(B_r)}\),\\
\| t^{1-2s} \pa_t \nabla_x U\|_{C^{\alpha}(Q_{r/2})} &\leq C\(\|\nabla_x U\|_{L^{\infty}(Q_r)} + \|U^{q-1} |\nabla_x U|\|_{C^1 (B_r)}\),\\
\| t^{2-2s} \pa_t^2 U\|_{C^{\alpha}(Q_{r/2})} &\leq  C \(\| t^{1-2s}\pa_t U\|_{C^{\alpha}(Q_{r/2})} + \|t^{2-2s} |\nabla_x^2 U|\|_{C^{\alpha}(Q_{r/2})}\)
\end{align*}
for some $\alpha \in (0,1)$.
\end{prop}
\begin{proof}
By Propositions 2.13 and 2.19 of \cite{JLX}, any positive solution $U$ to \eqref{eq-high-2} is twice differentiable in $x$ and it holds that
\begin{equation}\label{eq-high-0}
\begin{cases}
\text{div}(t^{1-2s} \nabla \pa_i U) =0 &\text{in } Q_r,\\
\pa_{\nu}^s (\pa_i U) = qU^{q-1} \pa_i U &\text{on } B_r
\end{cases}
\end{equation}
and
\begin{equation}\label{eq-high-3}
\begin{cases}
\text{div}(t^{1-2s} \nabla \pa_i \pa_j U) =0 &\text{in } Q_r,\\
\pa_{\nu}^s (\pa_i \pa_j U) = q U^{q-1} \pa_i \pa_j U + q(q-1) U^{q-2} (\pa_i U)(\pa_j U) &\text{on } B_r
\end{cases}
\end{equation}
for any $1 \le i,\, j \le N$.

\medskip
Let us prove validity of the estimates.
Applying Lemma \ref{lem-high-2} to equations \eqref{eq-high-2} and \eqref{eq-high-0}, we get
\begin{equation}\label{eq-high-1}
\int_{Q_{4r/5}}t^{1-2s} |\nabla \pa_i U|^2 dz \leq C\int_{Q_{5r/6}}t^{1-2s} |\pa_i U|^2 dz \leq C \int_{Q_r}t^{1-2s} |U|^2 dz \leq C\|U\|_{L^{\infty}(Q_r)}^2.
\end{equation}
Using this chain of inequalities and Lemma \ref{lem-high-2} once more, we find
\[\|\pa_i U\|_{L^{\infty}(Q_{3r/4})}^2 \leq C \int_{Q_{4r/5}} t^{1-2s}|\pa_i U|^2 dz \leq C\|U\|_{L^{\infty}(Q_r)}^2.\]
Hence Lemma \ref{lem-high-1} (1) gives  the first inequality of Proposition \ref{prop-high}
\[\|\pa_i U\|_{C^{\alpha}(Q_{r/2})} \leq C\(\|\pa_i U\|_{L^{\infty}(Q_{3r/4})} +  \|U^{q-1}\pa_i U\|_{L^{\infty}(B_{3r/4})}\)
\leq C \(1+ \|U^{q-1}\|_{L^{\infty}(B_r)}\)\| U\|_{L^{\infty}(Q_r)}.\]
Next, by employing Lemma \ref{lem-high-1} (2), we obtain the second inequality, i.e.,
\[\|t^{1-2s}\pa_t U\|_{C^{\alpha}(Q_{r/2})} \le C\(\|U\|_{L^{\infty}(Q_r)} + \|U^q\|_{C^{\beta}(B_r)}\).\]
Besides, an application of Lemma \ref{lem-high-2} to \eqref{eq-high-3} as well as inequality \eqref{eq-high-1} imply that
\begin{align*}
\|\pa_i \pa_j U\|_{L^{\infty}(Q_{3r/4})} &\leq C \(\int_{Q_{4r/5}} t^{1-2s} |\pa_i \pa_j U|^2 dz\)^{1/2} + C \|U^{q-2} (\pa_i U)(\pa_j U)\|_{L^{\infty} (B_{4r/5})} \\
&\leq C \(\| U\|_{L^{\infty}(Q_r)} + \|U^{q-2} (\pa_i U)(\pa_j U)\|_{L^{\infty} (B_r)}\).
\end{align*}
Therefore Lemma \ref{lem-high-1} (1) shows
\begin{align*}
\|\pa_i \pa_j U\|_{C^{\alpha}(Q_{r/2})} & \leq C \(\|\pa_i \pa_j U\|_{L^{\infty}(Q_{3r/4})} + \|U^{q-1}\pa_i \pa_j U\|_{L^{\infty}(B_{3r/4})} + \|U^{q-2} (\pa_i U) (\pa_j U)\|_{L^{\infty}(B_{3r/4})}\) \\
&\leq C \(\|U\|_{L^{\infty}(Q_r)} + \|U^{q-2} (\pa_i U)(\pa_j U)\|_{L^{\infty} (B_r)}\)
\\
&\ + C\|U^{q-1}\|_{L^{\infty}(B_r)} \(\|U\|_{L^{\infty}(Q_r)} + \|U^{q-2} (\pa_i U)(\pa_j U)\|_{L^{\infty} (B_r)}\),
\end{align*}
which is the third inequality of Proposition \ref{prop-high}.
On the other hand, by employing Lemma \ref{lem-high-1} (2) to \eqref{eq-high-0} again, we deduce the fourth inequality
\[\|t^{1-2s} \pa_t \pa_i U\|_{C^{\alpha}(Q_{r/2})} \leq C \(\|\pa_i U\|_{L^{\infty}(Q_r)} + \|U^{q-1} \pa_i U \|_{L^{\infty}(B_r)}\).\]
Finally, the last inequality follows from the fact that
\[t^{2-2s}\pa_t^2 U = -(1-2s)t^{1-2s} \pa_t U - t^{2-2s} \Delta_x U \quad \text{in } Q_{2r}.\]
This completes the proof.
\end{proof}

As a corollary of the above result, we get
\begin{cor}\label{cor-high}
Let $\{U_n\}_{n \in \mn}$ is a sequence of solutions of \eqref{eq-ext} with $\ep = \ep_n$.
For any $r>0$, let $A_r' = \mc' \setminus \cup_{i=1}^m \overline{B_{+}^{N+1} ((x_0^i, 0),r)}$.
Then there exists $\alpha \in (0,1)$ and a constant $C>0$ independent of $n \in \mn$ such that
\[\sum_{k=1}^2 \left\|\nabla_x^k \((\lambda_{n}^1)^{-\frac{N-2s}{2}} U_n\) \right\|_{C^{\alpha}(A_r')}
+ \sum_{\substack{0 \le k \le 1, 1 \le l \le 2, \\ 1 \le k+l \le 2}} \left\|t^{l-2s} \pa_t^l \nabla_x^k \((\lambda_{n}^1)^{-\frac{N-2s}{2}} U_n\) \right\|_{C^{\alpha}(A_r')} \le C\]
for any $n \in \mn$ large enough.
\end{cor}
\begin{proof}
Fix any compact subset $K \subset A_r'$ such that $K \cap \Omega \ne \emptyset$.
By \eqref{eq-u-ep-asym}, we have $\|U_n\|_{L^{\infty}(K)} \le C (\lambda_n^1)^{N-2s \over 2}$ (cf. Lemma \ref{lem-u-zero}).
Since Green's function $G_{\mc}$ is positive in $\mc$, again \eqref{eq-u-ep-asym} tells us that
the value $\inf_{z \in K}(\lambda_n^1)^{-{N-2s \over 2}} U_n(z)$ is bounded away from zero for large $n \in \mn$.
Thus even in the case that $p-2-\ep_n = {6-N \over N-2}-\ep_n < 0$ (i.e. $N \ge 6$),
we know
\[\left\| U^{p-2-\ep_n} |\nabla_x U|^2\right\|_{L^{\infty}(B_r)} \le C (\lambda_n^1)^{({N-2s \over 2})(p-\ep_n)}.\]
As a consequence,
\[\sum_{k=1}^2 \left\|\nabla_x^k U_n \right\|_{C^{\alpha}(A_r')}
+ \sum_{\substack{0 \le k \le 1, 1 \le l \le 2, \\ 1 \le k+l \le 2}} \left\|t^{l-2s} \pa_t^l \nabla_x^k U_n \right\|_{C^{\alpha}(A_r')} \le C (\lambda_{n}^1)^{\frac{N-2s}{2}}.\]
The proof is finished.
\end{proof}

\begin{proof}[Derivation of \eqref{eq-u-ep-asym-2} and \eqref{eq-u-ep-asym-3}]
Consider the sequence $\{\nabla_x U_n\}_{n \in \mn}$.
By Corollary \ref{cor-high}, it converges to some function $F$ uniformly over a compact subset of $A_r'$.
Then \eqref{eq-u-ep-asym} and an elementary analysis fact imply that $F = c_1 \sum_{i=1}^m b_i\, \nabla_x G_{\mc}((x,t),x_0^i)$. 
The other functions can be treated similarly.
This proves \eqref{eq-u-ep-asym-2} and \eqref{eq-u-ep-asym-3}.
\end{proof}

\end{document}